\documentclass[11pt,a4]{article}

\usepackage{amsthm,amssymb}
\usepackage{amsmath,yhmath}
\usepackage{mathrsfs}
\usepackage{hyperref,graphicx,color}
\usepackage[margin=3cm]{geometry}

\hypersetup{linkcolor=black,citecolor=black,colorlinks=true}

\theoremstyle{plain}
\newtheorem{theorem}{Theorem}[section]
\newtheorem{corollary}[theorem]{Corollary}
\newtheorem{lemma}[theorem]{Lemma}
\newtheorem{proposition}[theorem]{Proposition}

\theoremstyle{definition}

\theoremstyle{remark}
\newtheorem{remark}[theorem]{Remark}

\begin{document}

\title{Hardy Spaces ($1<p<\infty$) over Lipschitz Domains}

\author{Deng Guantie\thanks{
      E-mail: denggt@bnu.edu.cn, School of Mathematical Sciences, 
      Beijing Normal University, Beijing, China.} 
    and Liu Rong\thanks{
      Corresponding author, E-mail: rong.liu@mail.bnu.edu.cn
      School of Mathematical Sciences, 
      Beijing Normal University, Beijing, China.}
}
\maketitle

\begin{abstract}
  Let $\Gamma$ be a Lipschitz curve on the complex plane $\mathbb{C}$
  and $\Omega_+$ is the domain above $\Gamma$, we define Hardy space 
  $H^p(\Omega_+)$ as the set of holomorphic functions $F$ satisfying 
  $\sup_{\tau>0}(\int_{\Gamma} 
    |F(\zeta+\mathrm{i}\tau)|^p |\,\mathrm{d}\zeta|)^{\frac1p}< \infty$. 
  We mainly focus on the case of $1<p<\infty$ in this paper, 
  and prove that if $F(w)\in H^p(\Omega_+)$, 
  then $F(w)$ has non-tangential boundary limit $F(\zeta)$ a.e.\@ 
  on $\Gamma$, and $F(w)$ is the Cauchy integral of $F(\zeta)$. 
  We denote the conformal mapping from $\mathbb{C}_+$ onto $\Omega_+$ as $\Phi$, 
  and then prove that, $ H^p(\Omega_+)$ is isomorphic to $H^p(\mathbb{C}_+)$, 
  the classical Hardy space on upper half plane, under the mapping
  $T\colon F\to F(\Phi(z))\cdot (\Phi'(z))^\frac{1}{p}$, 
  where $F\in H^p(\Omega_+)$.
\end{abstract}



\section{Introduction}
Let $\Omega_+$ be a simply connected open set in the complex plane $\mathbb{C}$, 
and its boundary $\Gamma$ is an oriented locally rectifiable Jordan curve 
such that, $\Omega_+$ lies to the ``left'' of $\Gamma$. Denote the complement of 
$\overline{\Omega_+}$ as $\Omega_-$, then, for $0<p\leqslant\infty$, 
we could define two generalized Hardy spaces $H^p(\Omega_+)$ 
and $H^p(\Omega_-)$, both of which are subsets of holomorphic functions 
on the corresponding domains.

If $1<p<\infty$, as in the case where $\Omega_+$ is 
the upper complex plane $\mathbb{C}_+$, functions in these two spaces 
should have non-tangential boundary limit on $\Gamma$, 
which belongs to $L^p(\Gamma)$. We still denote these two sets 
of non-tangential boundary functions as $H^p(\Omega_+)$ and $H^p(\Omega_-)$, 
by abusing of language. One of the most famous problems about the 
generalized Hardy space, due to Calder\'on, states that, 
whether $L^p(\Gamma)$ is the direct sum of the spaces $H^p(\Omega_+)$ and 
$H^p(\Omega_-)$, for $1<p<\infty$. We could rewrite Calder\'on's problem as 
``$L^p(\Gamma)=H^p(\Omega_+)+H^p(\Omega_-)$?'' for simplicity, 
and the answer of this problem largely depends on the property of $\Gamma$,
based on what we have known for now.

In their book~\cite{01}, Meyer and Coifman provided three different proofs 
of the identity $L^2(\Gamma)=H^2(\Omega_+)+H^2(\Omega_-)$, 
where $\Gamma$ is regular in the sense of Ahlfors, 
except that they only considered boundary limit functions
from right above. The restriction on $\Gamma$, roughly speaking, 
requires that it could not have ``too many zigzags''.

If $0<p<\infty$, the usual way to define $F\in H^p(\Omega_+)$ is 
to require the boundedness of the integrals of $\lvert F\rvert^p$ 
over certain curves tending to the boundary $\Gamma$, and this is 
analogous to the definition of the classic Hardy spaces 
$H^p(\mathbb{C}_+)$ and $H^p(\mathbb{D})$. Here, $\mathbb{D}$ is the unit disk 
on $\mathbb{C}$. Since $H^p(\mathbb{C}_+)$ has been thoroughly studied, 
as is shown in books like \cite{03}~by Duren, \cite{06}~by Garnett, 
\cite{02}~by Deng, it would be convenient that $H^p(\Omega_+)$ could be 
connected with $H^p(\mathbb{C}_+)$ in a natural way.

In this paper, we will focus on the case of $1<p<\infty$, and prove that 
$H^p(\Omega_+)$ is isomorphic to $H^p(\mathbb{C}_+)$ under the mapping 
$T\colon F(w)\to F(\Phi(z))(\Phi'(z))^{\frac1p}$ under the assumption that 
$\Gamma$ is a Lipschitz curve. We will also prove that every function 
in $H^p(\Omega_+)$  is the Cauchy integral of its boundary limit, 
here the boundary limit functions are all non-tangential limits.
Our work is influenced by those of Meyer, Coifman and Duren, 
and provides the possibility to further the study of 
generalized Hardy space by utilizing results of 
the classic Hardy space $H^p(\mathbb{C}_+)$.

\section{Basic Definition}
Let $0<p\leqslant\infty$, and for $1\leqslant p\leqslant\infty$, 
denote $q$ as the conjugate coefficient of $p$, 
which means that $\frac{1}{p}+\frac{1}{q}=1$. 
We first introduce definitions of Hardy spaces over Lipschitz domains.

Let $a\colon\mathbb{R}\to\mathbb{R}$ be a Lipschitz function, 
where $|a(u_1)-a(u_2)|\leqslant M|u_1-u_2|$ 
for all $u_1$, $u_2\in\mathbb{R}$ and fixed $M>0$, 
$\Gamma=\{\zeta(u)=u+\mathrm{i}a(u)\colon u\in\mathbb{R}\}\subset\mathbb{C}$ 
be the graph of $a(u)$, we have $|a'(u)|\leqslant M$ a.e.\@, 
and the arc length measure of $\Gamma$ is 
$\mathrm{d}s= |\mathrm{d}\zeta|= (1+{a'}^2(u))^{1/2}\,\mathrm{d}u$. 
It follows that 
$\mathrm{d}u\leqslant \mathrm{d}s\leqslant (1+M^2)^{1/2}\mathrm{d}u$. 
For $F(\zeta)$ defined on $\Gamma$, since 
$\int_{\Gamma} |F(\zeta)|^p |\mathrm{d}\zeta|
  = \int_{\mathbb{R}} |F(u+\mathrm{i}a(u))|^p
    \cdot |1+\mathrm{i}a'(u)|\,\mathrm{d}u$, we have
\[\int_{\mathbb{R}} \big|F\big(u+\mathrm{i}a(u)\big)\big|^p \,\mathrm{d}u
  \leqslant \int_{\Gamma} |F(\zeta)|^p |\mathrm{d}\zeta|
  \leqslant \sqrt{1+M^2}\int_{\mathbb{R}} 
      \big|F\big(u+\mathrm{i}a(u)\big)\big|^p \,\mathrm{d}u,\]
that is, $F(\zeta)\in L^p(\Gamma,\mathrm{d}s)$ if and only if 
$F(u+\mathrm{i}a(u))\in L^p(\mathbb{R},\mathrm{d}u)$.
The space $L^p(\Gamma,\mathrm{d}s)$ may thus be identified with 
$L^p(\mathbb{R},\mathrm{d}u)$. We let $\Omega_+$ denote the set 
$\{u+\mathrm{i}v\in\mathbb{C} \colon v>a(u)\}$, $\Omega_-$ denote that 
$\{u+\mathrm{i}v\in\mathbb{C} \colon v<a(u)\}$, and $\Gamma_{\tau}$ denote 
$\Gamma+\mathrm{i}\tau= \{\zeta+\mathrm{i}\tau\colon \zeta\in\Gamma\}$ 
for $\tau\in\mathbb{R}$.

Let $F(w)$ be a function which is holomorphic on $\Omega_+$, 
we say that $F(w)\in H^p(\Omega_+)$, if
\[\sup_{\tau>0}\Big(\int_{\Gamma_\tau} |F(w)|^p\,\mathrm{d}s\Big)^{\frac1p}
  = \lVert F\rVert_{H^p(\Omega_+)}<\infty,\quad
  \text{for } 0<p<\infty,\]
or in the case of\/ $p=\infty$,
\[\sup_{w\in\Omega_+}|F(w)|=\lVert F\rVert_{H^\infty(\Omega_+)}<\infty.\]
Notice that 
\[\int_{\Gamma_\tau} |F(w)|^p\,\mathrm{d}s
  =\int_{\Gamma} |F(\zeta+\mathrm{i}\tau)|^p\,\mathrm{d}s.\]

Fix $u_0\in\mathbb{R}$ such that 
$\zeta'(u_0)=|\zeta'(u_0)|\mathrm{e}^{\mathrm{i}\phi_0}$ exists, and 
choose $\phi\in(0,\frac{\pi}2)$, we denote 
$\zeta_0=\zeta(u_0)=u_0+\mathrm{i}a(u_0)$ and let
\[\Omega_{\phi}(\zeta_0)
  =\{\zeta_0+r\mathrm{e}^{\mathrm{i}\theta}
     \colon r>0, \theta-\phi_0\in(\phi,\pi-\phi)\},\]
then we say that a function $F(w)$, defined on $\Omega_+$, 
has non-tangential boundary limit~$l$ at $\zeta_0$ if
\[\lim_{\substack{w\in\Omega_{\phi}(\zeta_0)\cap\Omega_+,\\ w\to\zeta_0}}
  F(w)= l, \quad \text{for any } \phi\in\Big(0,\frac{\pi}2\Big).\]
It is easy to verify that for fixed $\phi\in(0,\frac{\pi}2)$, 
there exists constant $\delta>0$, such that, if $|z|<\delta$ and 
$\zeta_0+z\in\Omega_\phi(\zeta_0)$, then $\zeta_0+z\in\Omega_+$ and 
$\zeta_0-z\in\Omega_-$.

It has been proved in book~\cite{01} that, 
every $F(w)\in H^p(\Omega_+)$ has upright down boundary limit a.e.\@ 
on $\Gamma$, and if we denote the limit function as $F(\zeta)$, 
then $F(\zeta)\in L^p(\Gamma)$, and
\[0=\frac1{2\pi\mathrm{i}} \int_{\Gamma} \frac{F(\zeta)}{\zeta-w}
  \,\mathrm{d}\zeta, \quad
  \text{for }w\in\Omega_-.\]
Besides, $F(w)$ is the Cauchy integral of $F(\zeta)$, that is 
\[F(w)=\frac1{2\pi\mathrm{i}} \int_{\Gamma} \frac{F(\zeta)}{\zeta-w}
    \,\mathrm{d}\zeta, \quad
    \text{for }w\in\Omega_+.\]
Actually, we could further prove that $F(\zeta)$ is 
the non-tangential boundary limit of $F(w)$, which will be shown in 
Corollary~\ref{cor-170622-2122}.

We then turn to definitions of the classical Hardy spaces over $\mathbb{C}_+$, 
the upper half complex plane. If $f(z)$ is holomorphic on 
$\mathbb{C}_+$ and $y>0$, we define
\[m(f,y)=\Big(\int_{\mathbb{R}} \lvert f(x+\mathrm{i}y)\rvert^p 
\,\mathrm{d}x\Big)^{\frac1p},\quad \text{if } 0<p<\infty,\]
or
\[m(f,y)= \sup_{x\in\mathbb{R}}\lvert f(x+\mathrm{i}y)\rvert,\quad 
\text{if } p=\infty,\]
then Hardy space $H^p(\mathbb{C}_+)$ is defined as
\[H^p(\mathbb{C}_+)
= \big\{f(z)\text{ is holomorphic on }\mathbb{C}_+\colon
\sup_{y>0} m(f,y)<\infty\big\},\]
and for $f(z)\in H^p(\mathbb{C}_+)$, define
\[\lVert f\rVert_{H^p(\mathbb{C}_+)}= \sup_{y>0} m(f,y).\]

Let $x_0\in\mathbb{R}$, $\alpha>0$ and 
$\Gamma_\alpha(x_0)=\{(x,y)\in\mathbb{C}_+\colon |x-x_0|<\alpha y\}$,
we say that a function $f(z)$, defined on $\mathbb{C}_+$, 
has non-tangential boudary limit~$l$ at $x_0$, if
\[\lim_{\substack{z\in\Gamma_\alpha(x_0),\\ z\to x_0}}f(z)=l,\quad
\text{for any }\alpha>0.\]
It is well-known that every function $f(z)\in H^p(\mathbb{C}_+)$, 
where $0<p\leqslant\infty$, has non-tangential boundary limit 
a.e.\@ on real axis. We usually denote the limit function as $f(x)$ 
for $x\in\mathbb{R}$, then $\lVert f\rVert_{H^p(\mathbb{C}_+)}
  = \lVert f\rVert_{L^p(\mathbb{R})}$.
If $1\leqslant p\leqslant\infty$ and $f(z)\in H^p(\mathbb{C}_+)$, 
then $f(z)$ is both the Cauchy and Poisson integral of $f(x)$~\cite{06}, 
that is, 
\[f(z)
  = \frac1{2\pi\mathrm{i}} \int_{\mathbb{R}}\frac{f(t)\,\mathrm{d}t}{t-z}
  = \frac1{\pi} \int_{\mathbb{R}} \frac{f(t)y\,\mathrm{d}t}{(x-t)^2+y^2},\]
for $z=x+\mathrm{i}y\in\mathbb{C}_+$, where $x\in\mathbb{R}$ and $y>0$.

By the definitions above, if $a(u)=0$, then $H^p(\Omega_+)= H^p(\mathbb{C}_+)$, 
thus we may consider $H^p(\mathbb{C}_+)$ as a special case of $H^p(\Omega_+)$. 
In this paper, we mainly focus on the case of $1<p<\infty$, then 
$H^p(\Omega_+)$ is a linear vector space equipped with 
the norm $\lVert\cdot\rVert_{H^p(\Omega_+)}$. From now on, 
we will assume that $1<p<\infty$, if not stated otherwise, 
thus $1<q<\infty$ too.

Let $\zeta$, $\zeta_0\in\Gamma$, we define
\[K_z(\zeta,\zeta_0)
  = \frac1{2\pi\mathrm{i}}\bigg(\frac1{\zeta-(\zeta_0+z)}
       - \frac1{\zeta-(\zeta_0-z)}\bigg),\]
for $z\in\mathbb{C}$ and $z\neq \pm(\zeta-\zeta_0)$, 
then $K_z(\zeta,\zeta_0)$ is well-defined and we could write 
\begin{equation}\label{equ-170623-1150}
  K_z(\zeta,\zeta_0)
  = \frac1{\pi\mathrm{i}}\cdot\frac{z}{(\zeta-\zeta_0)^2-z^2},
\end{equation}
We could also verify that, if $\zeta_0+z\in\Omega_+$ and 
$\zeta_0-z\in\Omega_-$, then 
\[\int_{\Gamma} K_z(\zeta,\zeta_0)\,\mathrm{d}\zeta= 1.\]

Since $\Omega_+$ is an open and simply-connected subset of the complex plane, 
there exists a conformal representation $\Phi$ from $\mathbb{C}_+$ 
onto $\Omega_+$, which extends to an increasing homeomorphism of 
$\mathbb{R}$ onto $\Gamma$, that is $\mathrm{Re\,}\Phi'(x)>0$, 
for $x\in\mathbb{R}$ a.e.\@. Define 
$\Sigma= \{x+\mathrm{i}y\in\mathbb{C}\colon 
    x>0, \lvert y\rvert\leqslant Mx\}$, then
$\Phi'(x)\in \overline{\Sigma}$ a.e.\@.

Denote the inverse of $\Phi(z)$ as $\Psi(w)\colon \Omega_+\to\mathbb{C}_+$, 
then 
\[\Phi\big(\Psi(w)\big)=w, \quad\text{for all } w\in\Omega_+,\]
and
\[\Psi\big(\Phi(z)\big)=z, \quad\text{for all } z\in\mathbb{C}_+.\]
Besides, if $w=\Phi(z)$ for $z\in\overline{\mathbb{C}_+}$, 
then $\Phi'(z)\cdot\Psi'(w)=1$, 
thus $\mathrm{Re}\,\Psi'(\zeta)>0$ for $\zeta\in\Gamma$ a.e.\@.
More details about $\Phi$ are in Lemma~\ref{lem-170706-2130} 
and Lemma~\ref{lem-170522-1205}.

We now consider a transform $T$ 
from $H^p(\Omega_+)$ to holomorphic functions on $\mathbb{C}_+$, where
\begin{equation}\label{equ-170602-1910}
  TF(z)= F\big(\Phi(z)\big)\cdot \big(\Phi'(z)\big)^\frac{1}{p}, \quad
    \text{for } F\in H^p(\Omega_+),
\end{equation}
then, obviously, $T$ is linear and one-to-one. The main result of this paper 
is the following theorem.

\begin{theorem}\label{thm-170601-1721}
  If $T$ is defined as above, 
  then $T\colon H^p(\Omega_+)\to H^p(\mathbb{C}_+)$ 
  is linear, one-to-one, onto and bounded. Its inverse 
  $T^{-1}\colon H^p(\mathbb{C}_+)\to H^p(\Omega_+)$ is also bounded.
\end{theorem}

\section{Some useful Lemmas}
We introduce some lemmas about $K_z(\zeta,\zeta_0)$ first.
\begin{lemma}\label{lem-170622-2050}
  There exists a constant $C>0$, such that
  \[|K_{\mathrm{i}\tau}(\zeta,\zeta_0)|
    \leqslant \frac{C\tau}{|\zeta-\zeta_0|^2+\tau^2},\quad
    \text{for } \tau>0.\]
\end{lemma}
\begin{proof}
  By equation~\eqref{equ-170623-1150}, we have 
  \[K_{\mathrm{i}\tau}(\zeta,\zeta_0)
    = \frac1{\pi}\cdot\frac{\tau}{(\zeta-\zeta_0)^2+\tau^2},\]
  then we only need to prove that there exists a constant $C'>0$, such that 
  \begin{equation}\label{equ-170623-1155}
    |\zeta-\zeta_0|^2+\tau^2\leqslant C'|(\zeta-\zeta_0)^2+\tau^2|.
  \end{equation}
  
  If $|\zeta-\zeta_0|\geqslant \frac{3\tau}{2}$, then
  \begin{align*}
  |\zeta-\zeta_0|^2+\tau^2
  &\leqslant |\zeta-\zeta_0|^2+\frac49|\zeta-\zeta_0|^2
  = \frac{13}{9}|\zeta-\zeta_0|^2,                  \\
  |(\zeta-\zeta_0)^2+\tau^2|
  &\geqslant |\zeta-\zeta_0|^2-\tau^2
  \geqslant \frac59|\zeta-\zeta_0|^2,
  \end{align*}
  or if $|\zeta-\zeta_0|\leqslant \frac{2\tau}{3}$, then
  \begin{align*}
  |\zeta-\zeta_0|^2+\tau^2
  &\leqslant \frac49\tau^2+\tau^2
  = \frac{13}{9}\tau^2,                  \\
  |(\zeta-\zeta_0)^2+\tau^2|
  &\geqslant \tau^2-|\zeta-\zeta_0|^2
  \geqslant \frac59\tau^2.
  \end{align*}
  In both cases, we could choose $C'\geqslant\frac{13}{5}$, 
  and the inequality~\eqref{equ-170623-1155} holds.
  
  If $\frac{2\tau}{3}<|\zeta-\zeta_0|<\frac{3\tau}{2}$, then 
  $|\zeta-\zeta_0|^2+\tau^2< \frac{13}{4}\tau^2$. 
  Suppose $\zeta=u+\mathrm{i}a(u)$, $\zeta_0=u_0+\mathrm{i}a(u_0)$, 
  where $u$, $u_0\in\mathbb{R}$, we have 
  \[\frac{2\tau}{3}
    <\big|u-u_0+\mathrm{i}\big(a(u)-a(u_0)\big)\big|
    <\frac{3\tau}{2},\]
  then $\frac{2\tau}{3}<\sqrt{1+M^2}|u-u_0|$ since 
  $|a(u)-a(u_0)|\leqslant M|u-u_0|$. Now that
  \begin{align*}
    &|(\zeta-\zeta_0)^2+\tau^2|\\
    ={} &\bigl|(u-u_0)^2-\big(a(u)-a(u_0)\big)^2 +\tau^2
      + 2\mathrm{i}(u-u_0)\big(a(u)-a(u_0)\big)\bigr|,
  \end{align*}
  if $|a(u)-a(u_0)|\geqslant\frac{\tau}{2}$, then 
  \begin{align*}
    |(\zeta-\zeta_0)^2+\tau^2|
    &\geqslant 2|u-u_0|\cdot|a(u)-a(u_0)|    \\
    &\geqslant \frac{2\tau^2}{3\sqrt{1+M^2}},
  \end{align*}
  and if $|a(u)-a(u_0)|<\frac{\tau}{2}$, then
  \begin{align*}
    |(\zeta-\zeta_0)^2+\tau^2|
    &\geqslant \big|(u-u_0)^2-\big(a(u)-a(u_0)\big)^2 +\tau^2\big|   \\
    &= (u-u_0)^2-\big(a(u)-a(u_0)\big)^2 +\tau^2                   \\
    &> \frac{3\tau^2}{4}.
  \end{align*}
  In either case, if we choose $C'\geqslant\frac{39\sqrt{1+M^2}}{8}$, then
  $|\zeta-\zeta_0|^2+\tau^2\leqslant C'|(\zeta-\zeta_0)^2+\tau^2|$.
  
  Thus, let $C'=\max\{\frac{13}5,\frac{39\sqrt{1+M^2}}{8}\}$, then 
  \[|K_{\mathrm{i}\tau}(\zeta,\zeta_0)|
    \leqslant \frac{C'}{\pi}\cdot\frac{\tau}{|\zeta-\zeta_0|^2+\tau^2},\]
  and the lemma is proved.
\end{proof}

By the proof above, the constant $C$ only depends on $M$, 
and we actually have
\[|K_{\mathrm{i}\tau}(\zeta,\zeta_0)|
  \leqslant \frac{C|\tau|}{|\zeta-\zeta_0|^2+\tau^2},\quad
  \text{for } \tau\in\mathbb{R}.\]
\begin{corollary}\label{cor-170622-2210}
  If $F(\zeta)\in L^p(\Gamma,|\mathrm{d}\zeta|)$, and we define
  \[F(w)
    = \int_\Gamma K_{\mathrm{i}\tau}(\zeta,\zeta_0) F(\zeta)
      \,\mathrm{d}\zeta,\]
  for $w=\zeta_0+\mathrm{i}\tau\in\Omega_+$, 
  where $\zeta_0\in\Gamma$ and $\tau>0$, then 
  \[\sup_{\tau>0} \int_{\Gamma_\tau} |F(w)|^p |\mathrm{d}w|< \infty.\]
\end{corollary}
\begin{proof}
  Fix $\tau>0$ in $w=\zeta_0+\mathrm{i}\tau$, 
  and let $\zeta_0=u_0+\mathrm{i}a(u_0)$, where $u_0\in\mathbb{R}$. 
  By Lemma~\ref{lem-170622-2050}, we have 
  \begin{align*}
    |F(w)|=|F(\zeta_0+\mathrm{i}\tau)|
    &\leqslant \int_{\Gamma} |F(\zeta)| \frac{C\tau|\mathrm{d}\zeta|}
        {|\zeta-\zeta_0|^2+\tau^2}                            \\
    &\leqslant \int_{\mathbb{R}} \big|F\big(u+\mathrm{i}a(u)\big)\big| 
        \frac{C\tau\sqrt{1+M^2}\,\mathrm{d}u}{|u-u_0|^2+\tau^2},
  \end{align*}
  then
  \begin{align*}
    &\int_{\Gamma_\tau} |F(w)|^p |\mathrm{d}w|               \\
    \leqslant{}& \int_{\mathbb{R}} 
        \bigg(\int_{\mathbb{R}} |F(u+\mathrm{i}a(u))| 
        \frac{C\tau\sqrt{1+M^2}\,\mathrm{d}u}{|u-u_0|^2+\tau^2}\bigg)^p
        \sqrt{1+M^2}\,\mathrm{d}u_0                            \\
    ={}& (1+M^2)^{\frac12+\frac{p}2}C^p \int_{\mathbb{R}} 
       \bigg(\int_{\mathbb{R}} \big|F\big(u+u_0+\mathrm{i}a(u+u_0)\big)\big| 
          \frac{\tau\,\mathrm{d}u}{u^2+\tau^2}\bigg)^p \mathrm{d}u_0.
  \end{align*}
  By Minkowski's inequality, 
  \begin{align*}
    &\bigg(\int_{\Gamma_\tau} |F(w)|^p |\mathrm{d}w|\bigg)^{\frac1p}  \\
    \leqslant{}& (1+M^2)^{\frac1{2p}+\frac12}C \int_{\mathbb{R}} 
        \bigg(\int_{\mathbb{R}} \big|F\big(u+u_0+\mathrm{i}a(u+u_0)\big)\big|^p 
           \mathrm{d}u_0 \bigg)^{\frac1p}
        \frac{\tau\,\mathrm{d}u}{u^2+\tau^2}                          \\
    ={}& (1+M^2)^{\frac1{2p}+\frac12}C \int_{\mathbb{R}} 
       \bigg(\int_{\mathbb{R}} \big|F\big(u_0+\mathrm{i}a(u_0)\big)\big|^p 
           \mathrm{d}u_0 \bigg)^{\frac1p}
       \frac{\mathrm{d}u}{u^2+1}                                    \\
    \leqslant{}& (1+M^2)^{\frac1{2p}+\frac12}C\pi  
       \bigg(\int_{\Gamma} |F(\zeta_0)|^p|\mathrm{d}\zeta_0|\bigg)^{\frac1p}\\
    ={}& (1+M^2)^{\frac1{2p}+\frac12}C\pi 
       \lVert F\rVert_{L^p(\Gamma,|\mathrm{d}\zeta|)},
  \end{align*}
  and this proves the corollary.
\end{proof}

The above corollary is obviously true if $p=1$.

\begin{lemma}\label{lem-170629-2230}
  Suppose $\zeta_0=\zeta(u_0)\in\Gamma$, $\phi\in(0,\frac\pi2)$ are both fixed, 
  and $\zeta'(u_0)$ exists, then we could choose positive constants 
  $C$ and $\delta$, depending on $\phi$ and $\zeta_0$, respectively, 
  such that if $z+\zeta\in\Omega_\phi(\zeta_0)\cap\Omega_+$ 
  and $|z|<\delta$, then
  \[|K_z(\zeta,\zeta_0)|
    \leqslant \frac{C|z|}{|\zeta-\zeta_0|^2+|z|^2}.\]
\end{lemma}
\begin{proof}
  We only need to prove that there exists constants $C'$, $\delta>0$, 
  such that 
  \[|\zeta-\zeta_0|^2+|z|^2\leqslant C'|(\zeta-\zeta_0)^2-z^2|,\]
  if $z+\zeta\in\Omega_\phi(\zeta_0)\cap\Omega_+$ and $|z|<\delta$.
  Let $\phi_1=\frac12\phi$. Since
  \[\lim_{u\to u_0}\frac{\zeta(u)-\zeta(u_0)}{u-u_0}
    = \zeta'(u_0)= |\zeta'(u_0)|\mathrm{e}^{\mathrm{i}\phi_0},\]
  where $\phi_0\in(-\frac\pi2,\frac\pi2)$, then there exists $\delta>0$, 
  depending on $\zeta_0$, such that for all $|u-u_0|< 2\delta$, we have 
  \[\arg(\zeta-\zeta_0)-\phi_0
    \in (-\phi_1,\phi_1)\cup(\pi-\phi_1,\pi+\phi_1).\]
  In this case, we could let $z=|z|\mathrm{e}^{\mathrm{i}\theta}$, 
  where $\theta-\phi_0\in(\phi,\pi-\phi)$, then
  \[\theta-\arg(\zeta-\zeta_0)
    \in (-\pi+\phi_1,-\phi_1)\cup (\phi_1,\pi-\phi_1),\]
  and 
  \begin{align*}
  |\zeta-\zeta_0+z|
  &= \big||\zeta-\zeta_0|\mathrm{e}^{\mathrm{i}\arg(\zeta-\zeta_0)}
  - |z|\mathrm{e}^{\mathrm{i}\theta}\big|           \\
  &= \big||\zeta-\zeta_0|- |z|\mathrm{e}^{
    \mathrm{i}\theta- \mathrm{i}\arg(\zeta-\zeta_0)}\big|     \\
  &\geqslant |z|\cdot|\sin(\theta-\arg(\zeta-\zeta_0))|           \\
  &\geqslant |z|\sin\phi_1.
  \end{align*}
  We also have $|\zeta-\zeta_0-z|\geqslant |z|\sin\phi_1$ and 
  $|\zeta-\zeta_0\pm z|\geqslant |\zeta-\zeta_0|\sin\phi_1$
  by using the same method.
  
  If $|\zeta-\zeta_0|\leqslant |z|$, then
  $|\zeta-\zeta_0|^2+|z|^2\leqslant 2|z|^2$, and
  \[|(\zeta-\zeta_0)^2-z^2|
    =|\zeta-\zeta_0+z|\cdot |\zeta-\zeta_0-z|
    \geqslant |z|^2\sin^2\phi_1,\]
  thus
  \[|\zeta-\zeta_0|^2+|z|^2
    \leqslant \frac{2}{\sin^2\phi_1}|(\zeta-\zeta_0)^2-z^2|.\]
  If $|\zeta-\zeta_0|> |z|$, then
  $|\zeta-\zeta_0|^2+|z|^2\leqslant 2|\zeta-\zeta_0|^2$, and
  \[|(\zeta-\zeta_0)^2-z^2|\geqslant |\zeta-\zeta_0|^2\sin^2\phi_1,\] 
  we still have
  \[|\zeta-\zeta_0|^2+|z|
    \leqslant \frac{2}{\sin^2\phi_1}|(\zeta-\zeta_0)^2-z^2|.\]
    
  In the case of $|u-u_0|\geqslant 2\delta$, we know that 
  $|\zeta-\zeta_0|\geqslant |u-u_0|\geqslant 2\delta$.
  For $|z|<\delta\leqslant \frac12|\zeta-\zeta_0|$, we have 
  \[|\zeta-\zeta_0|^2+|z|^2\leqslant \frac54|\zeta-\zeta_0|^2,\]
  and
  \[|(\zeta-\zeta_0)^2-z^2|
    \geqslant |\zeta-\zeta_0|^2-|z|^2
    \geqslant \frac34|\zeta-\zeta_0|^2,\]
  thus
  \[|\zeta-\zeta_0|^2+|z|^2
    \leqslant \frac53|(\zeta-\zeta_0)^2-z^2|.\]
  
  In both cases, if we let $C'=\max\{\frac53,\frac2{\sin^2\phi_1}\}$, 
  then for all $|z|<\delta$, we have
  \[|\zeta-\zeta_0|^2+|z|^2
    \leqslant C'|(\zeta-\zeta_0)^2-z^2|,\]
  and this proves the lemma.
\end{proof}

The following corollary was proved in \cite{05}, but the proof here is new 
and simpler.

\begin{corollary}\label{cor-170622-2122}
  If $F(\zeta)\in L^p(\Gamma,|\mathrm{d}\zeta|)$, and $u_0$ is 
  the Lebesgue point of $F(u+\mathrm{i}a(u))$ such that 
  $\zeta'(u_0)= |\zeta'(u_0)|\mathrm{e}^{\mathrm{i}\phi_0}$ exists,
  where $\phi_0\in(-\frac\pi2,\frac\pi2)$, 
  then for any $\phi\in(0,\frac{\pi}2)$, we have
  \[\lim_{\substack{z+\zeta_0\in\Omega_\phi(\zeta_0)\cap\Omega_+,\\ 
        z\to 0}}
    \int_\Gamma K_z(\zeta,\zeta_0)F(\zeta)\,\mathrm{d}\zeta
    = F(\zeta_0).\]
\end{corollary}
\begin{proof}
  Fix $\phi\in(0,\frac\pi2)$. 
  If $z+\zeta_0\in\Omega_\phi(\zeta_0)\cap\Omega_+$, and 
  $|z|$ is sufficiently small, then $z-\zeta_0\in\Omega_-$ 
  and $\int_\Gamma K_z(\zeta,\zeta_0)\,\mathrm{d}\zeta=1$. 
  By Lemma~\ref{lem-170629-2230}, we have
  \begin{align*}
       &\left|\int_\Gamma K_z(\zeta,\zeta_0)F(\zeta)\,\mathrm{d}\zeta
         - F(\zeta_0)\right|                             \\
    ={}& \left|\int_\Gamma K_z(\zeta,\zeta_0) 
        \big(F(\zeta)- F(\zeta_0)\big)\,\mathrm{d}\zeta\right|   \\
    \leqslant{}& \int_\Gamma \frac{C|z|\cdot |F(\zeta)-F(\zeta_0)|}
        {|\zeta-\zeta_0|^2+|z|}|\mathrm{d}\zeta|       \\
    \leqslant{}& C\sqrt{1+M^2}\int_{\mathbb{R}} 
        \big|F\big(\zeta(u)\big)-F\big(\zeta(u_0)\big)\big|
        \frac{|z|\,\mathrm{d}u}{|u-u_0|^2+|z|^2}          \\
    ={}& \pi C\sqrt{1+M^2} \int_{\mathbb{R}} 
        \big|F\big(\zeta(u)\big)-F\big(\zeta(u_0)\big)\big|
        P_{|z|}(u-u_0)\,\mathrm{d}u                            \\
    ={}& \pi C\sqrt{1+M^2} \int_{\mathbb{R}} 
        \big|F\big(\zeta(u+u_0)\big)-F\big(\zeta(u_0)\big)\big|
        P_{|z|}(u)\,\mathrm{d}u,
  \end{align*}
  where $C>0$ is constant, depending on $\phi$, and 
  $P_y(x)=\frac1\pi\cdot\frac{y}{x^2+y^2}$ is the Poisson kernel 
  on $\mathbb{C}_+$. Since $u_0$ is the Lebesgue point of $F(\zeta(u))$, 
  we have~\cite{02}
  \[\lim_{|z|\to 0} \int_{\mathbb{R}} 
       \big|F\big(\zeta(u+u_0)\big)-F\big(\zeta(u_0)\big)\big|
       P_{|z|}(u)\,\mathrm{d}u =0,\]
  thus $\int_\Gamma K_z(\zeta,\zeta_0)F(\zeta)\,\mathrm{d}\zeta
  \to F(\zeta_0)$ as $z\to 0$.
\end{proof}

The above corollary shows that, for $w=\zeta_0+z\in\Omega_+$ 
where $\zeta_0\in\Gamma$ and $z\in\mathrm{C}$, if we define 
$G(w)= G(\zeta_0+z)
  =\int_\Gamma K_z(\zeta,\zeta_0)F(\zeta)\,\mathrm{d}\zeta$, 
then $G(w)$ has non-tangential boundary limit $F(\zeta_0)$ at $\zeta_0$,
although $G(w)$ maybe only defined in $\Omega_\phi(\zeta_0)$ 
and near $\zeta_0$.
\begin{corollary}\label{cor-170706-2100}
  If $F(w)\in H^p(\Omega_+)$ and $F(\zeta)$ is the upright down 
  boundary limit of $F(w)$ on $\Gamma$, then $F(\zeta)$ is also
  the non-tangential boundary limit of $F(w)$.
\end{corollary}
\begin{proof}
  Let $u_0$ be the Lebesgue point of $F(\zeta(u))$ and 
  $\zeta'(u_0)$ exists, $\zeta_0=\zeta(u_0)$ 
  and $\phi\in(0,\frac\pi2)$ fixed. Since there exists $\delta>0$ 
  such that if $z+\zeta_0\in\Omega_\phi(\zeta_0)\cap\Omega_+$ 
  and $|z|<\delta$, then $\zeta_0-z\in\Omega_-$, we have 
  \[F(z+\zeta_0)
    =\frac1{2\pi\mathrm{i}} \int_{\Gamma} 
       \frac{F(\zeta)\,\mathrm{d}\zeta}{\zeta-(\zeta_0+z)},\quad
    0= \frac1{2\pi\mathrm{i}} \int_{\Gamma} 
       \frac{F(\zeta)\,\mathrm{d}\zeta}{\zeta-(\zeta_0-z)},\]
  and then
  \begin{align*}
    F(z+\zeta_0)
    &= \frac1{2\pi\mathrm{i}} \int_{\Gamma} F(\zeta)\bigg(
          \frac{1}{\zeta-\zeta_0-z}- \frac{1}{\zeta-\zeta_0+z}
       \bigg)\,\mathrm{d}\zeta                         \\
    &= \int_\Gamma K_z(\zeta,\zeta_0)F(\zeta)\,\mathrm{d}\zeta.
  \end{align*}
  Corollary~\ref{cor-170622-2122} shows that
  \[\lim_{z\to 0} F(z+\zeta_0)= F(\zeta_0),\]
  which means that $F(w)$ tends to $F(\zeta)$ non-tangentially.
\end{proof}

Thus, every function in $H^p(\Omega_+)$ is the Cauchy integral 
of its non-tangential boundary limit.

\begin{lemma}\label{lem-170622-2200}
  If $F(w)$ is a rational function, 
  vanishing at infinity, whose poles do not lie in $\overline{\Omega_+}$,
  then $F(w)\in H^p(\Omega_+)$.
\end{lemma}
\begin{proof}  
  We consider a simple case first. 
  Let $F(w)=\frac{1}{w-\alpha}$ with 
  $\alpha=\alpha_1+\mathrm{i}\alpha_2\in\Omega_-$ 
  where $\alpha_1$, $\alpha_2\in\mathbb{R}$,
  then $\alpha_1+\mathrm{i}a(\alpha_1)\in\Gamma$, 
  and $d=|\alpha_2-a(\alpha_1)|>0$. Let $w=u+\mathrm{i}v\in\Gamma_\tau$, 
  where $u$, $v\in\mathbb{R}$. Since the slopes of $\Gamma$ are between 
  $-M$ and $M$, we have 
  \[\int_{\Gamma_\tau} |F(w)|^p |\mathrm{d}w|
    \leqslant \sqrt{1+M^2} \int_\mathbb{R} \frac{\mathrm{d}u}
        {\big((u-\alpha_1)^2+(v-\alpha_2)^2\big)^{\frac{p}{2}}},\]
  and if $|u-\alpha_1|\leqslant \frac{d}{M}$, then 
  \[|w-\alpha|\geqslant \frac{d}{\sqrt{1+M^2}},\]
  
  Since
  \begin{align*}
    &\int_{|u-\alpha_1| \leqslant\frac{d}{M}} \frac{\,\mathrm{d} u}
        {\big((u-\alpha_1)^2+(v-\alpha_2)^2\big)^{\frac{p}{2}}}   \\
    \leqslant{} &\int_{|u-\alpha_1| \leqslant\frac{d}{M}} 
           \Big(\frac{d}{\sqrt{1+M^2}}\Big)^{-p}\,\mathrm{d} u  \\
    ={} &2d^{1-p}M^{-1}(1+M^2)^{\frac{p}{2}},
  \end{align*}
  and
  \begin{align*}
    &\int_{|u-\alpha_1|>\frac{d}{M}} \frac{\,\mathrm{d} u}
       {\big((u-\alpha_1)^2+(v-\alpha_2)^2\big)^{\frac{p}{2}}}    \\
    \leqslant{} &\int_{|u-\alpha_1| >\frac{d}{M}} 
     \frac{\,\mathrm{d} u}{|u-\alpha_1|^p}                  \\
    ={} &\frac2{p-1}d^{1-p}M^{p-1},
  \end{align*}
  we get that
  \[\int_{\Gamma_\tau} |F(w)|^p |\mathrm{d}w|
    \leqslant d^{1-p}M^{-1}\Big(2(1+M^2)^{\frac{p}{2}}
        + \frac2{p-1}M^p\Big),\]
  which shows that $F(w)\in H^p(\Omega_+)$.
  
  If $F(w)=\frac1{(w-\alpha)^k}$ where $k$ is an integer greater than $1$, 
  then we also have $F(w)\in H^p(\Omega_+)$ since 
  $\frac1{w-\alpha}\in H^{pk}(\Omega_+)$.
  
  For the general case, we could write $F(w)$ as
  \[\sum_{j=1}^{N_1}\sum_{k=1}^{N_2} \frac{c_{jk}}{(w-\alpha_j)^k}.\]
  then, by what we have proved, 
  $(\int_{\Gamma_\tau}|F(w)|^p\,\mathrm{d}s)^{\frac1p}$ is bounded above
  by a constant which is independent of $\tau$. Thus, we still have 
  $F(w)\in H^p(\Omega_+)$.
\end{proof}

Kenig proved the following lemma in \cite{07}.
\begin{lemma}\label{lem-170706-2130}
  {\rm (i)} $\Phi$ extends to $\overline{\mathbb{C}_+}$ as a homeomorphism 
    onto $\overline{\Omega_+}$.
    
  {\rm (ii)} $\Phi'$ has a non-tangential limit almost everywhere, and this limit
  is different from $0$ almost everywhere.
  
  {\rm (iii)} $\Phi$ preserves angles at almost every boundary point, 
  and so does $\Psi$ (here almost every refers to $|\mathrm{d}\zeta|$).
  
  {\rm (iv)} $\Phi(x)$, $x\in\mathbb{R}$, is absolutely continuous 
  when restricted to any finite interval, and hence $\Phi'(x)$ exists 
  almost everywhere, and is locally integrable, moreover, 
  for almost every $x\in\mathbb{R}$, $\Phi'(x) = \lim_{z\to x} \Phi'(z)$,
  where $z\in\mathbb{C}_+$ converges nontangentially to $x$.
  
  {\rm (v)} Sets of measure $0$ on $\mathbb{R}$ correspond to sets 
  of measure $0$ (with respect to $|\mathrm{d}\zeta|$) on $\Gamma$, 
  and vice versa.
  
  {\rm (vi)} At every point where $\Phi'(x)$ exists and is different from $0$, 
  it is a vector tangent to the curve $\Gamma$ at the point $\Phi(x)$, 
  and hence $|\arg\Phi'(x)|\leqslant \arctan M$ for almost every $x$.
\end{lemma}

The following lemma shows that, for each $y>0$, 
the curves $\Phi(x+\mathrm{i}y)$, $x\in\mathbb{R}$, 
are graphs of Lipschitz functions whose slopes are between $-M$ and $M$. 
We let $\theta_0= \arctan M$ for simplicity.

\begin{lemma}\label{lem-170522-1205}
  For the conformal representation $\Phi\colon \mathbb{C}_+\to\Omega_+$,
  we have
  	\[\mathrm{Re}\,\Phi'(z)>0, \quad  
      \text{and } 
      \lvert\mathrm{Im}\,\Phi'(z)\rvert\leqslant M\mathrm{Re}\,\Phi'(z),\]
  for all $z\in\mathbb{C}_+$.
\end{lemma}

\begin{remark}
  Although the proof of this lemma was inspired by \cite{01}, we made 
  some small but important adjustments and the proof of the lemma, 
  which contains much more details, is complete.
\end{remark}

\begin{proof}
	We start with the special case where $\Gamma$ is a polygonal line 
	ending with two half-lines whose slopes are between 
  $-M$ and $M$, and let $N$ be the number of vertices of $\Gamma$. Then by 
  Schwarz-Christoffel formula~\cite{04}, there exists $N$ real numbers 
  $c_1<c_2<\cdots<c_N$, a real number~$\gamma$, 
  and $N$ real exponents $\gamma_j$, $1\leqslant j\leqslant N$, such that 
  \[\Phi'(z)
    = \mathrm{e}^{\mathrm{i}\gamma} \prod_{j=1}^{N}(z-c_j)^{\gamma_j}.\]
  The ``angles'' $\gamma$ and $\gamma_j$ are related to the slopes of 
  the polygonal line $\Gamma$, and we choose a branch of $\log\Phi'(z)$ 
  such that, $\arg(z-c_j)\in(0,\pi)$ for $z\in\mathbb{C}_+$, 
  and if $x<c_1$, then 
  \[\mathrm{Im}\log\Phi'(x)= \gamma+\pi(\gamma_1+\cdots+\gamma_N);\] 
  if $c_j<x<c_{j+1}$, $j=1$, $\ldots$, $N-1$, then 
  \[\mathrm{Im}\log\Phi'(x)= \gamma+\pi(\gamma_{j+1}+\cdots+\gamma_N);\] 
  if $x>c_N$, then 
  $\mathrm{Im}\log\Phi'(x)= \gamma$.
  Besides, in each case, 
  $\lvert \mathrm{Im}\log\Phi'(x)\rvert\leqslant \tan^{-1}M= \theta_0$.
  
  Since
  \[\log\Phi'(z)= \mathrm{i}\gamma+\sum_{j=1}^{N}\gamma_j\log(z-c_j),\]
  we have
  \[\mathrm{Im}\log\Phi'(z)
    = \gamma+\sum_{j=1}^{N}\gamma_j\mathrm{Im\,}\log(z-c_j),\]
  then
  \[\lvert \mathrm{Im}\log\Phi'(z)\rvert
    \leqslant \lvert\gamma\rvert+\pi\sum_{j=1}^N\lvert\gamma_j\rvert.\]
    
  The above inequality shows that the harmonic function 
  $\mathrm{Im}\log\Phi'(z)$ is bounded on $\mathbb{C}_+$. 
  By the Poisson representation of bounded harmonic function 
  on $\mathbb{C}_+$~\cite{02},
  \[\mathrm{Im}\log\Phi'(z)
    = \frac{1}{\pi}\int_{-\infty}^{+\infty} 
        \frac{y\mathrm{Im}\log\Phi'(t)}{(t-x)^2+y^2}\,\mathrm{d}t,\quad
    \text{for }z=x+\mathrm{i}y\in\mathbb{C}_+,\]
  then
  \begin{align*}
    \lvert \mathrm{Im}\log\Phi'(z)\rvert
    &\leqslant \frac{1}{\pi}\int_{-\infty}^{+\infty} 
       \frac{y\lvert\mathrm{Im}\log\Phi'(t)\rvert\,\mathrm{d}t}{(t-x)^2+y^2}   \\
    &\leqslant \frac{\theta_0}{\pi} \int_{-\infty}^{+\infty} 
       \frac{y\,\mathrm{d}t}{(t-x)^2+y^2}
     = \theta_0,
  \end{align*}
  and this proves that $|\arg\Phi'(z)|\leqslant \theta_0$ and 
  $\lvert\mathrm{Im}\,\Phi'(z)\rvert\leqslant 
    M\lvert\mathrm{Re}\,\Phi'(z)\rvert$. Since 
  \[\mathrm{Re\,}\Phi'(z)= |\Phi'(z)|\cos\arg\Phi'(z),\]
  we have $\mathrm{Re\,}\Phi'(z)>0$.

  We then pass to the general case by approximating $\Gamma$ by 
  a sequence $\Gamma_j$ of polygonal lines such that the open sets $\Omega_{j+}$ 
  above $\Gamma_j$ increase to $\Omega_+$, and it is enough that $\Gamma_j$ are 
  graphs of piecewise affine function $a_j(x)$ which decrease to $a(x)$.

  To construct the $a_j(x)$, we consider the subdivision of $\mathbb{R}$ 
  formed by the points $x=k\cdot 2^{-j}$, where $k\in\mathbb{Z}$ and 
  $\lvert k\rvert\leqslant j\cdot 2^j$. Let
  $y(k,j)= a(k\cdot 2^{-j})+2M\cdot 2^{-j}$, and denote $\Gamma_j$ as the 
  polygonal line whose nodes are $(k\cdot 2^{-j}, y(k,j))$, 
  and whose ends are formed by half-lines of slope $-M$ and $M$, respectively.
  That is,

  If $x\leqslant -j$, then $a_j(x)= a(-j)+2M\cdot 2^{-j}-M(x+j)$;

  If $k\cdot 2^{-j}< x\leqslant (k+1)\cdot 2^{-j}$, 
  where $-j\cdot 2^j\leqslant k< j\cdot 2^j$, then 
  \[a_j(x)= 2^j\big(y(k+1,j)-y(k,j)\big)(x-k\cdot 2^{-j})+y(k,j);\]

  If $x>j$, then $a_j(x)= a(j)+2M\cdot 2^{-j}+M(x-j)$.

  Let $K\subset\mathbb{R}$ be a compact set and choose $j$ large enough 
  such that $K\subset[-j,j]$, then for $x\in[k\cdot 2^{-j}, (k+1)\cdot 2^{-j}]$,
  \begin{align*}
  \lvert a_j(x)-a(x)\rvert
  &\leqslant a(k\cdot 2^{-j})+2M\cdot 2^{-j}+M(x-k\cdot 2^{-j})     \\
    &\phantom{\leqslant{}}- \big(a(k\cdot 2^{-j})-M(x-k\cdot 2^{-j})\big)       \\
  &= 2M\cdot 2^{-j}+2M(x-k\cdot 2^{-j})          \\
  &\leqslant 4M\cdot 2^{-j}.
  \end{align*}
  It follows that the functions $a_j(x)$ form a sequence which decreases, 
  uniformly on each compact set of $\mathbb{R}$ to $a(x)$, and we could choose 
  a subsequence of $\{\Omega_{j+}\}$, which is denoted as $\{\Omega_{j+}\}$ again, 
  such that, 
  \[\Omega_{1+}\subset\Omega_{2+}\subset\cdots\subset\Omega_+, \text{ and } 
    \Omega_+= \bigcup_{j=1}^{+\infty} \Omega_{j+}.\]

  Let $\Phi_j(z)$ be the conformal representation from $\mathbb{C}_+$ onto 
  $\Omega_{j+}$, where $\Phi_j(\infty)= \infty$, 
  $\Phi_j(\mathrm{i})=w_0\in\Omega_{1+}$, 
  then by Carath\'eodory convergence theorem, 
  $\Phi_j(z)$ converge uniformly on each compact set of 
  $\mathbb{C}_+$ to the conformal representation 
  $\Phi\colon \mathbb{C}_+\to\Omega_+$. Since 
  $\lvert \mathrm{Im\,}\log\Phi_j'(z)\rvert\leqslant\theta_0$, 
  for $z\in\mathbb{C}_+$, we have 
  $\lvert \mathrm{Im\,}\log\Phi'(z)\rvert\leqslant\theta_0$, that is 
  $\lvert\mathrm{Im}\,\Phi'(z)\rvert\leqslant 
      M\lvert\mathrm{Re}\,\Psi'(z)\rvert$. Besides, 
  $\mathrm{Re\,}\Phi'(z)\geqslant 0$ as $\mathrm{Re\,}\Phi_j'(z)\geqslant 0$.
  But, as a harmonic function, if $\mathrm{Re\,}\Phi'(z_0)=0$ 
  for some $z_0\in\mathrm{C}_+$, we should have $\mathrm{Re\,}\Phi'(z)=0$
  for all $z\in\mathrm{C}_+$. Then $\mathrm{Im\,}\Phi'(z)=0$ 
  and $\Phi'(z)=0$, which is impossible.
  
  Thus, we have proved that $\mathrm{Re}\,\Phi'(z)>0$ and 
  $\lvert\mathrm{Im}\,\Phi'(z)\rvert\leqslant M\mathrm{Re}\,\Phi'(z)$,
  for all $z\in\mathbb{C}_+$.
\end{proof}

\section{Proof of the Main Theorem}
We divide the proof of the main theorem into three parts, 
and deal with the ``onto'' part first.

\begin{lemma}\label{lem-170522-1245}
  Suppose that $1<q<\infty$, $\alpha\in\mathbb{C}$, $\varepsilon>0$, 
  $E(\alpha,\varepsilon)=\{z\in\mathbb{C}_+\colon
      |\Phi(z)-\alpha|\geqslant \varepsilon\}$.
  Let $E_y=\{t\in\mathbb{R}\colon t+\mathrm{i}y\in E(\alpha,\varepsilon)\}$ 
  for $y>0$, then
  \[I= \int_{E_y} \frac{\lvert\Phi'(t+\mathrm{i}y)\rvert\,\mathrm{d}t}
         {\lvert\Phi(t+\mathrm{i}y)-\alpha\rvert^q}
     \leqslant \frac{2^{q+1}\sqrt{1+M^2}}{(q-1)\varepsilon^{q-1}}.\]

  As a consequence, if $\alpha\notin\overline{\Omega_+}$, and we define
  \[g(z)=\frac{\big(\Phi'(z)\big)^{\frac1q}}{\Phi(z)-\alpha},\quad
    \text{for } z\in\mathbb{C}_+,\]
  then $g\in H^q(\mathbb{C}_+)$.
\end{lemma}

\begin{proof}
  Since $\lvert\Phi(z)-\alpha\rvert\geqslant \varepsilon$ 
  for $z\in E(\alpha,\varepsilon)$, then 
  \[\lvert \Phi(z)-\alpha)\rvert
    \geqslant \frac12\big(\lvert \mathrm{Re}\,\Phi(z)
          - \mathrm{Re}\,\alpha\rvert+ \varepsilon\big).\]
  By Lemma~\ref{lem-170522-1205}, we have $\mathrm{Re}\,\Phi'(z)>0$, 
  that is, $\mathrm{Re}\,\Phi(t+\mathrm{i}y)$ 
  is an increasing function of $t$, then
  \begin{align*}
    I
    &= \int_{E_y} \frac{\lvert\Phi'(t+\mathrm{i}y)\rvert\,\mathrm{d}t}
           {\lvert\Phi(t+\mathrm{i}y)-\alpha\rvert^q}          \\
    &\leqslant \int_{E_y} \frac{
          \sqrt{1+M^2}\,\mathrm{d}\,\mathrm{Re}\,\Phi(t+\mathrm{i}y)}
        {2^{-q}\big(\lvert \mathrm{Re}\,\Phi(t+\mathrm{i}y)
          -\mathrm{Re}\,\alpha\rvert+ \varepsilon\big)^q}         \\
    &\leqslant \int_{\mathbb{R}}\frac{\sqrt{1+M^2}\,\mathrm{d}t}
        {2^{-q}(\lvert t\rvert+\varepsilon)^q}                   \\ 
    &= \frac{2^{q+1}\sqrt{1+M^2}}{(q-1)\varepsilon^{q-1}}.
  \end{align*}

  If $\alpha\notin\overline{\Omega_+}$, there exists $\varepsilon>0$, 
  such that $\lvert\Phi(z)-\alpha\rvert\geqslant \varepsilon$ 
  for all $z\in\mathbb{C}_+$. Then $E_y=\mathbb{R}$, and
  \[\int_{\mathbb{R}} \lvert g(t+\mathrm{i}y)\rvert^q \,\mathrm{d}t
    = \int_{\mathbb{R}} \frac{\lvert\Phi'(t+\mathrm{i}y)\rvert\,\mathrm{d}t}
           {\lvert\Phi(t+\mathrm{i}y)-\alpha\rvert^q}
    \leqslant \frac{2^{q+1}\sqrt{1+M^2}}{(q-1)\varepsilon^{q-1}},\]
  thus $g\in H^q(\mathbb{C}_+)$ as the boundary above 
  is independent of $y$, and the lemma is proved.
\end{proof}

To show that $T(H^p(\Omega_+))$ 
contains $H^p(\mathbb{C}_+)$, we need the following two lemmas.

\begin{lemma}\label{lem-170522-1424}
  If $F(w)\in H^p(\Omega_+)$ and $G(w)\in H^q(\Omega_+)$, where 
  $1<p<\infty$ and $\frac1p+\frac1q=1$, $F(\zeta)$ and $G(\zeta)$
  are the nontangential boundary limit function of $F(w)$ and $G(w)$,
  respectively, then
  \[\int_\Gamma F(\zeta)G(\zeta)\,\mathrm{d}\zeta=0.\]
\end{lemma}

The proof of the above lemma is nearly the same as in \cite{01}, 
so we omit it here.

\begin{lemma}\label{lem-170522-1300}
  If $F(\zeta)\in L^p(\Gamma,\lvert\mathrm{d}\zeta\rvert)$, 
  then $F(\zeta)$ is the non-tangential boundary limit of a function in $H^p(\Omega_+)$ 
  if and only if
  \[\int_{\Gamma}\frac{F(\zeta)}{\zeta-\alpha}\,\mathrm{d}\zeta=0, \quad
    \text{for all } \alpha\notin\overline{\Omega_+}.\]
\end{lemma}
\begin{proof}
  ``$\Rightarrow$'': it is obvious if we combine 
  Lemma~\ref{lem-170622-2200} and Lemma~\ref{lem-170522-1424}.
  
  ``$\Leftarrow$'': for $w\in\Omega_+$, define 
  \[F(w)= \frac1{2\pi\mathrm{i}} 
       \int_{\Gamma}\frac{F(\zeta)}{\zeta-w}\,\mathrm{d}\zeta.\]
  If $w_1\in\Omega_+$, then there exists constant $\delta>0$ such that
  the open disk $D(w_1,2\delta)\subset\Omega_+$. 
  Choose $w_2\in D(w_1,\delta)$, then for $\zeta\in\Gamma$,
  \[|\zeta-w_2|\geqslant |\zeta-w_1|-\delta\geqslant \frac12|\zeta-w_1|,\]
  and
  \begin{align*}
    |F(w_1)-F(w_2)|
    &\leqslant \frac1{2\pi} \int_\Gamma 
          \frac{|(w_1-w_2)F(\zeta)|}{|\zeta-w_1||\zeta-w_2|}
          |\mathrm{d}\zeta|                                   \\
    &\leqslant \frac{|w_1-w_2|}{\pi} \int_\Gamma 
          \frac{|F(\zeta)||\mathrm{d}\zeta|}{|\zeta-w_1|^2}   \\
    &\leqslant \frac{|w_1-w_2|}{\pi} \bigg(
        \int_\Gamma |F(\zeta)|^p|\mathrm{d}\zeta|\bigg)^{\frac1p}
        \bigg(\int_\Gamma \frac{|\mathrm{d}\zeta|}{|\zeta-w_1|^{2q}}
           \bigg)^{\frac1q}                              \\
    &\leqslant \frac{|w_1-w_2|}{\pi} 
        \lVert F\rVert_{L^p(\Gamma,|\mathrm{d}\zeta|)}\cdot I.
  \end{align*}
  We could use the same method as in Lemma~\ref{lem-170622-2200} to prove
  that $I$ is bounded by a constant depending only on $w_1$. 
  It follows that
  \[\lim_{w_2\to w_1} F(w_2)= F(w_1).\]
  Now we have proved that $F(w)$ is continous on $\Omega_+$, and it is 
  an easy consequence of Morera's theorem that $F(w)$ is actually 
  analytic on $\Omega_+$.
  
  If we write $w=\zeta_0+\mathrm{i}\tau$ where $\zeta_0\in\Gamma$ 
  and $\tau>0$, then $\zeta_0-\mathrm{i}\tau\in\Omega_-$, and
  \begin{align*}
    F(w)
    &= \frac1{2\pi\mathrm{i}} \int_{\Gamma} F(\zeta) \Big(
         \frac1{\zeta-(\zeta_0+\mathrm{i}\tau)}
         - \frac1{\zeta-(\zeta_0-\mathrm{i}\tau)}\Big)\mathrm{d}\zeta   \\
    &= \int_\Gamma F(\zeta)K_{\mathrm{i}\tau}(\zeta,\zeta_0) 
          \,\mathrm{d}\zeta.
  \end{align*}
  By Corollary~\ref{cor-170622-2210}, $F(w)\in H^p(\Omega_+)$.
  
  For fixed $\zeta_0\in\Gamma$ and $\phi\in(0,\frac{\pi}2)$, if 
  $w\in\Omega_\phi(\zeta_0)\cap\Omega_+$, 
  we write $w=\zeta_0+z$, then there exist $\delta>0$, such that 
  $w_0-z\in\Omega_-$ for all $|z|<\delta$, and
  \begin{align*}
    F(w)
    &= \frac1{2\pi\mathrm{i}} \int_{\Gamma} F(\zeta) \Big(
         \frac1{\zeta-(\zeta_0+z)}- \frac1{\zeta-(\zeta_0-z)}\Big)
         \mathrm{d}\zeta   \\
    &= \int_\Gamma F(\zeta)K_z(\zeta,\zeta_0) 
         \,\mathrm{d}\zeta.
  \end{align*}
  By Corollary~\ref{cor-170622-2122}, $F(w)\to F(\zeta_0)$ if $w\to\zeta_0$, 
  that is $F(w)$ has non-tangential boudary limit $F(\zeta_0)$ 
  at $\zeta_0\in\Gamma$.
  
  Thus, $F(\zeta)$ is the non-tangential boundary limit function of
  $F(w)\in H^p(\Omega_+)$.
\end{proof}

Now we reach our first part of the main theorem's proof.

\begin{proposition}\label{pro-170522-1450}
  For $T$ defined by~\eqref{equ-170602-1910}, we have 
  \[H^p(\mathbb{C}_+)\subset T\big(H^p(\Omega_+)\big), \quad
    \text{or } T^{-1}\big(H^p(\mathbb{C}_+)\big)\subset H^p(\Omega_+).\]
\end{proposition}

\begin{proof}
  Fix $f(z)\in H^p(\mathbb{C}_+)$, 
  let $F(w)= f(\Psi(w))(\Psi'(w))^{\frac1p}$ for $w\in\Omega_+$,
  then $F(w)$ is holomorphic on $\Omega_+$,
  \[TF(z)
    = F\big(\Phi(z)\big)\big(\Phi'(z)\big)^{\frac1p}
    = f(z),\]
  since $\Phi'(z)\cdot\Psi'(w)=1$ if $\Phi(z)=w$, and $F(w)$ has 
  non-tangential boundary limit $F(\zeta)$ on $\Gamma$. 
  Then for $\alpha\notin\overline{\Omega_+}$, 
  \begin{align*}
    \int_{\Gamma} \frac{F(\zeta)}{\zeta-\alpha}\,\mathrm{d}\zeta
    &= \int_{\Gamma}
         \frac{f\big(\Psi(\zeta)\big)\big(\Psi'(\zeta)\big)^{\frac1p}}
              {\zeta-\alpha}\,\mathrm{d}\zeta                     \\
    &= \int_{\mathbb{R}}
         \frac{f(t)\big(\Phi'(t)\big)^{\frac1q}}{\Phi(t)-\alpha}\,\mathrm{d}t.  
  \end{align*}
  The second equation holds since $\Psi(\zeta)=t\in\mathbb{R}$ 
  for $\zeta\in\Gamma$, and $\Psi'(\zeta)\cdot\Phi'(t)=1$. 
  
  Define
  \[g(z)= \frac{\big(\Phi'(z)\big)^{\frac1q}}{\Phi(z)-\alpha}, \quad
    \text{for } z\in\mathbb{C}_+,\]
  then $g\in H^q(\mathbb{C_+})$, by Lemma~\ref{lem-170522-1245}, and 
  $g(z)$ has non-tangential boundary limit $g(t)$ on $\mathbb{R}$. 
  Thus, by Lemma~\ref{lem-170522-1424},
  \[\int_{\Gamma} \frac{F(\zeta)}{\zeta-\alpha}\,\mathrm{d}\zeta
    = \int_{\mathbb{R}} f(t)g(t) \,\mathrm{d}t
    = 0.\]
  By Lemma~\ref{lem-170522-1300}, $F(\zeta)$ is the non-tangential 
  boundary limit of a function in $H^p(\Omega_+)$, 
  and we let the function be $G(w)$. 
  In the following, we are going to prove that $G(w)=F(w)$.
  
  Since, by Lemma~\ref{lem-170522-1300}, 
  \[G(w)
    =\frac{1}{2\pi\mathrm{i}}\int_{\Gamma}
        \frac{F(\zeta)}{\zeta-w}\,\mathrm{d}\zeta,\]
  then
  \[G(w)
    = \frac{1}{2\pi\mathrm{i}}\int_{\Gamma}
        \frac{f\big(\Psi(\zeta)\big)\big(\Psi'(\zeta)\big)^{\frac1p}}
             {\zeta-w}\,\mathrm{d}\zeta
    = \frac{1}{2\pi\mathrm{i}}\int_{\mathbb{R}}
        \frac{f(t)\big(\Phi'(t)\big)^{\frac1q}}
             {\Phi(t)-w}\,\mathrm{d}t.\]
  From 
  \[F(w)
    = f\big(\Psi(w)\big)\big(\Psi'(w)\big)^{\frac1p}
    = \frac{1}{2\pi\mathrm{i}} \int_{\mathbb{R}}
         \frac{f(t)\,\mathrm{d}t}{t-\Psi(w)} \big(\Psi'(w)\big)^{\frac1p},\]
  we have
  \[F(w)-G(w)
    = \frac{1}{2\pi\mathrm{i}} \int_{\mathbb{R}} f(t) 
         \bigg(\frac{\big(\Psi'(w)\big)^{\frac1p}}{t-\Psi(w)}
            - \frac{\big(\Phi'(t)\big)^{\frac1q}}{\Phi(t)-w}\bigg)
          \,\mathrm{d}t.\]
  Fix $w_0\in\Omega_+$, let $\Psi(w_0)= z_0= x_0+\mathrm{i}y_0$, 
  where $x_0\in\mathbb{R}$, $y_0>0$, and define
  \[h(z)
    = \frac{\big(\Psi'(w_0)\big)^{\frac1p}}{z-\Psi(w_0)}
      - \frac{\big(\Phi'(z)\big)^{\frac1q}}{\Phi(z)-w_0},\]
  for $z\in\mathbb{C}_+\setminus\{z_0\}$, then $h(z)$ is holomorphic, 
  has non-tangential boundary limit $h(t)$ on real axis, and
  \[h(z)
    = \frac{\big(\Phi'(z_0)\big)^{-\frac1p}}{z-z_0}
         - \frac{\big(\Phi'(z)\big)^{\frac1q}}{\Phi(z)-\Phi(z_0)} 
    = h_1(z)- h_2(z).\]
  Since
  \begin{align*}
    \lim_{z\to z_0} h(z)
    &= \lim_{z\to z_0} \frac{
            \big(\Phi'(z_0)\big)^{-\frac1p}\big(\Phi(z)-\Phi(z_0)\big)
            - (z-z_0)\big(\Phi'(z)\big)^{\frac1q}}
          {(z-z_0)^2\Phi'(z_0)}                                   \\
    &= \lim_{z\to z_0} \frac{\big(\Phi'(z_0)\big)^{-\frac1p}\Phi'(z)
            - \big(\Phi'(z)\big)^{\frac1q}
            - (z-z_0)\cdot\frac1q\big(\Phi'(z)\big)^{-\frac1p}\Phi''(z)}
          {2\Phi'(z_0)(z-z_0)}                                     \\
    &= \frac{\big(\Phi'(z_0)\big)^{\frac1q}}{2\big(\Phi'(z_0)\big)^{1+\frac1p}}
       \lim_{z\to z_0} \frac{
          \big(\Phi'(z)\big)^{\frac1p}- \big(\Phi'(z_0)\big)^{\frac1p}}
          {z-z_0}
       - \frac{\Phi''(z_0)}{2q\big(\Phi'(z_0)\big)^{1+\frac1p}}     \\
    &= \frac{1}{2\big(\Phi'(z_0)\big)^{1+\frac1p}}
       \bigg(\big(\Phi'(z_0)\big)^{\frac1q}
          \cdot \frac1p\big(\Phi'(z_0)\big)^{\frac1p-1}\cdot \Phi''(z_0)
          - \frac1q\Phi''(z_0)\bigg)                                 \\
    &= \frac12\Big(\frac1p-\frac1q\Big) 
       \frac{\Phi''(z_0)}{\big(\Phi'(z_0)\big)^{1+\frac1p}},
  \end{align*}
  it follows that $h(z)$ could be extended holomorphically to $\mathbb{C}_+$.

  We are going to show that $h(z)\in H^q(\mathbb{C}_+)$, for then
  $\int_{\mathbb{R}} f(t)h(t)\,\mathrm{d}t= 0$, and $F(w_0)=G(w_0)$.
  Choose $\delta_0>0$ small enough, such that 
  \[E_0=\{x+\mathrm{i}y\in\mathbb{C}_+\colon 
      \lvert x-x_0\rvert\leqslant \delta_0, \lvert y-y_0\rvert\leqslant \delta_0\}
    \subset \mathbb{C}_+,\]
  and write
  \begin{align*}
    I
    &= \int_{\mathbb{R}} \lvert h(t+\mathrm{i}y)\rvert^q\,\mathrm{d}t      \\
    &= \int_{\{t\colon \lvert t-x_0\rvert>\delta_0\}} 
          \lvert h(t+\mathrm{i}y)\rvert^q\,\mathrm{d}t
       + \int_{\{t\colon \lvert t-x_0\rvert\leqslant\delta_0\}} 
           \lvert h(t+\mathrm{i}y)\rvert^q\,\mathrm{d}t        \\
    &= I_1+ I_2.
  \end{align*}
  We then have
  \begin{align*}
    \int_{\lvert t-x_0\rvert>\delta_0} 
      \lvert h_1(t+\mathrm{i}y)\rvert^q\,\mathrm{d}t
    &= \big\lvert\Phi'(z_0)\big\rvert^{-\frac{q}{p}}
       \cdot \int_{\{t\colon \lvert t-x_0\rvert>\delta_0\}} 
          \frac{\mathrm{d}t}{\lvert t+\mathrm{i}y-z_0\rvert^q}         \\
    &\leqslant 2\big\lvert\Phi'(z_0)\big\rvert^{-\frac{q}{p}}
       \int_{x_0+\delta_0}^{+\infty} \frac{\mathrm{d}t}{(t-x_0)^q}            \\
    &= \frac{2}{(q-1)\delta_0^{q-1}}
       \big\lvert\Phi'(z_0)\big\rvert^{-\frac{q}{p}}.
  \end{align*}
  If $\lvert y-y_0\rvert>\delta_0$, then 
  \begin{align*}
    \int_{\{t\colon \lvert t-x_0\rvert\leqslant\delta_0\}} 
        \lvert h_1(t+\mathrm{i}y)\rvert^q\,\mathrm{d}t    
    &\leqslant \big\lvert\Phi'(z_0)\big\rvert^{-\frac{q}{p}} 
        \int_{\{t\colon \lvert t-x_0\rvert\leqslant\delta_0\}} 
          \frac{\mathrm{d}t}{\lvert y-y_0\rvert^q}                \\
    &\leqslant \big\lvert\Phi'(z_0)\big\rvert^{-\frac{q}{p}} 
        \frac{2\delta_0}{\delta_0^q}
      = \frac{2}{\delta_0^{q-1}} \big\lvert\Phi'(z_0)\big\rvert^{-\frac{q}{p}}.
  \end{align*}

  Denote $E_y=\{t\in\mathbb{R}\colon t+\mathrm{i}y\notin E_0\}$ for fixed $y$.
  Since $\Phi(E_0)\subset\Omega_+$, there exists $\varepsilon>0$, such that  
  $\lvert \Phi(z)-\Phi(z_0)\rvert\geqslant\varepsilon_1$, 
  for $z\in\mathbb{C}_+\setminus E_0$. 
  By Lemma~\ref{lem-170522-1245},
  \[\int_{E_y} \lvert h_2(t+\mathrm{i}y)\rvert^q\,\mathrm{d}t
    = \int_{E_y} \frac{\lvert\Phi'(t+\mathrm{i}y)\rvert\,\mathrm{d}t}
        {\lvert\Phi(t+\mathrm{i}y)-\Phi(z_0)\rvert^q}      
    \leqslant \frac{2^{q+1}\sqrt{1+M^2}}{(q-1)\varepsilon_1^{q-1}}.\]

  Thus, by Minkowski's inequality, $I_1$ for all $y>0$, 
  and $I_2$ for $\lvert y-y_0\rvert>\delta_0$, are bounded, 
  and the boundaries are independent of $y$. Since $h(z)$ is continuous 
  on the compact set $E_0$, we could denote
  $M_1=\max\{\lvert h(z)\rvert\colon z\in E_0\}$, 
  then for $\lvert y-y_0\rvert\leqslant \delta_0$,
  \[I_2
    = \int_{\lvert t-x_0\rvert\leqslant\delta_0} 
         \lvert h(t+\mathrm{i}y)\rvert^q\,\mathrm{d}t
    \leqslant M_1^q\cdot 2\delta_0.\]
  Thus, $I$ is bounded, independent of $y$, for all $y>0$, 
  and $h(z)\in H^q(\mathbb{C}_+)$. By Lemma~\ref{lem-170522-1424}, 
  \[F(w_0)-G(w_0)
    = \frac{1}{2\pi\mathrm{i}} \int_{\mathbb{R}} f(t)h(t)\,\mathrm{d}t
    = 0,\]
  and $F(w)=G(w)\in H^p(\Omega_+)$, which proves the theorem.
\end{proof}

Thus we could define $T^{-1}$, the inverse of $T$, on $H^p(\mathbb{C}_+)$ as 
\[T^{-1}f(w)= f\big(\Psi(w)\big)\big(\Psi'(w)\big)^{\frac1p},\quad
  \text{for } w\in\Omega_+.\]

The analogies of Lemma~\ref{lem-170522-1205} and Lemma~\ref{lem-170522-1245} 
are the following two lemmas.

\begin{lemma}\label{lem-170511-0936}
	If holomorphic representation $\Psi(w)\colon \Omega_+\to\mathbb{C}_+$
  is the inverse of $\Phi(z)$, then 
  \[\mathrm{Re\,}\Psi'(w)>0,\quad\text{and } 
    \lvert\mathrm{Im\,}\Psi'(w)\rvert\leqslant M\mathrm{Re\,}\Psi'(w)\]
  for all $w\in\Omega_+$.
\end{lemma}

\begin{proof}
  Let $\Phi(z)=w\in\Omega_+$, then $z\in\mathbb{C}_+$, 
  $\Phi'(z)\cdot\Psi'(w)=1$, and 
  \[\Psi'(w)
    = \frac{1}{\Phi'(z)}
    = \frac{\overline{\Phi'(z)}}{\lvert\Phi'(z)\rvert^2},\]
  \[\mathrm{Re}\,\Psi'(w)
    = \frac{\mathrm{Re}\,\Phi'(z)}{\lvert\Phi'(z)\rvert^2}
    > 0, \quad
    \mathrm{Im}\,\Psi'(w)
    = \frac{-\mathrm{Im}\,\Phi'(z)}{\lvert\Phi'(z)\rvert^2},\]
  by Lemma~\ref{lem-170522-1205}, thus
  \[\frac{\lvert\mathrm{Im\,}\Psi'(w)\rvert}{\mathrm{Re\,}\Psi'(w)}
    = \frac{\lvert\mathrm{Im\,}\Phi'(z)\rvert}{\mathrm{Re\,}\Phi'(z)}
    \leqslant M,\]
  and the lemma is proved.
\end{proof}
Again, we have $\Psi'(w)\in\Sigma$ for all $w\in\Omega_+$.

\begin{lemma}\label{lem-170522-1536}
  Suppose $1<q<\infty$, $\alpha\in\mathbb{C}$, $\varepsilon>0$, 
  $E(\alpha,\varepsilon)=\{w\in\Omega_+\colon
      |\Psi(w)-\alpha|>\varepsilon\}$. 
  Let $E_{\tau}= \{u\in\mathbb{R}\colon u+\mathrm{i}a(u)+\mathrm{i}\tau
    \in E(\alpha,\varepsilon)\}$ for $\tau>0$, then 
  \[I= \int_{E_{\tau}}
       \frac{\big\lvert\Psi'\big(u+\mathrm{i}a(u)+\mathrm{i}\tau\big)
         	     \big(1+\mathrm{i}a'(u)\big)\big\rvert}
            {\big\lvert\Psi\big(u+\mathrm{i}a(u)+\mathrm{i}\tau\big)
             	 - \alpha\big\rvert^q}\,\mathrm{d}u\]
  is bounded, and the boundary is independent of $\tau$.

  Consequently, if $\alpha\notin\overline{\mathbb{C}_+}$, and we define
  \[G(w)=\frac{\big(\Psi'(w)\big)^{\frac1q}}{\Psi(w)-\alpha},\quad
    \text{for } w\in\Omega_+,\]
  then $G\in H^q(\Omega_+)$.
\end{lemma}

\begin{proof}
  Remember that $\Sigma=\{(x,y)\in\mathbb{C}\colon x>0, 
    \lvert y\rvert\leqslant Mx\}$. We then divide $\Sigma$ into 
  $N$ equally parts by drawing $N-1$ half-lines from the origin, such that 
  the angle between two adjacent half-lines is 
  $\frac{2\theta_0}{N}<\frac12(\frac{\pi}{2}-\theta_0)$, 
  and denote that angle as $\theta_1$. That is, 
  $\Sigma=\bigcup_{j=1}^{N} \Sigma_j$, and for $1\leqslant j\leqslant N$, 
  \[\Sigma_j
    =\big\{r\mathrm{e}^{\mathrm{i}\theta}\in\mathbb{C}\colon 
       r>0,\ \theta-\bigl(-\theta_0+(j-1)\theta_1\bigr)\in[0,\theta_1]\big\}.\]

  Denote $E_{\tau j}=\{u\in E_{\tau}\colon
     \arg\Psi'(u+ \mathrm{i}a(u)+ \mathrm{i}\tau)\in\Sigma_j\}$,
  for $j=1$, $\ldots$, $N$. Then we write 
  \begin{align*}
    I
    &= \sum_{j=1}^{N} \int_{E_{\tau j}}
         \frac{\big\lvert\Psi'\big(u+\mathrm{i}a(u)+\mathrm{i}\tau\big)
         	       \big(1+\mathrm{i}a'(u)\big)\big\rvert}
              {\big\lvert\Psi\big(u+\mathrm{i}a(u)+\mathrm{i}\tau\big)
               	 - \alpha\big\rvert^q} \,\mathrm{d}u                      \\
    &= \sum_{j=1}^{N} I_{\tau j}.
  \end{align*}

  Now we fix $j$, and let $h_j(u)= \Psi(u+\mathrm{i}a(u)+\mathrm{i}\tau)\cdot
    \mathrm{e}^{\mathrm{i}(\theta_0-(j-1)\theta_1)}$, for $u\in E_{\tau j}$, then 
  \[\frac{\mathrm{d}h_j}{\mathrm{d}u}
    = \Psi'\big(u+\mathrm{i}a(u)+\mathrm{i}\tau\big)\big(1+\mathrm{i}a'(u)\big)
      \cdot \mathrm{e}^{\mathrm{i}(\theta_0-(j-1)\theta_1)},\]
  which follows that 
  \[\arg\frac{\mathrm{d}h_j}{\mathrm{d}u}
    \in [-\theta_0,\theta_0+\mathrm{i}\theta_1]
    \subset \Bigl(-\frac{\pi}{2},\frac{\pi}{2}\Bigr),\]
  thus $\mathrm{Re}\,\frac{\mathrm{d}h_j}{\mathrm{d}u}> 0$, 
  and $\mathrm{Re}\,h_j(u)$ is an increasing function of $u$. Besides,
  \[\Big\lvert\frac{\mathrm{d}h_j}{\mathrm{d}u}\Big\rvert
    \leqslant \sqrt{1+\tan^2(\theta_0+\theta_1)}
      \mathrm{Re}\,\frac{\mathrm{d}h_j}{\mathrm{d}u}
    = M_1 \mathrm{Re}\,\frac{\mathrm{d}h_j}{\mathrm{d}u}.\]

  Let $\alpha\mathrm{e}^{\mathrm{i}(\theta_0-(j-1)\theta_1)}
    = z_{0j}= x_{0j}+\mathrm{i}y_{0j}$, where $x_{0j}$, $y_{0j}\in\mathbb{R}$. 
  Since $\lvert\Psi(w)-\alpha\rvert\geqslant \varepsilon$ for $w\in E$, then 
  \begin{align*}
    \big\lvert\Psi\big(u+\mathrm{i}a(u)+\mathrm{i}\tau\big)- \alpha\big\rvert
    &= \lvert h_j(u)- z_{0j}\rvert                 \\
    &\geqslant \frac12\big(\lvert\mathrm{Re}\, h_j(u)-x_{0j}\rvert
          +\varepsilon\big),
  \end{align*}
  and we have
  \begin{align*}
    I_{\tau j}
    &= \int_{E_{\tau j}} \frac{\lvert\mathrm{d}h_j(u)\rvert}
          {\lvert h_j(u)-z_{0j}\rvert^q}                                \\
    &\leqslant \int_{\mathbb{R}} \frac{M_1\,\mathrm{d\,\big(Re}\,h_j(u)\big)}
        {2^{-q}\big(\lvert\mathrm{Re}\, h_j(u)-x_{0j}\rvert
           + \varepsilon\big)^q}                                \\
    &\leqslant \int_{\mathbb{R}} \frac{M_1\,\mathrm{d}t}
        {2^{-q}(|t|+\varepsilon)^q}                             \\
    &= \frac{2^{q+1} M_1}{(q-1)\varepsilon^{q-1}},
  \end{align*}
  and then
  \[I
    = \sum_{j=1}^{N} I_{\tau j}
    \leqslant \frac{2^{q+1} M_1 N}{(q-1)\varepsilon^{q-1}}.\]

  If $\alpha\notin\overline{\mathbb{C}_+}$, there exists $\varepsilon>0$, 
  such that $\lvert\Psi(w)-\alpha\rvert> \varepsilon$ 
  for all $w\in\Omega_+$. Then $E_{\tau}=\mathbb{R}$, and
  \[\int_{\Gamma} \lvert G(\zeta+\mathrm{i}\tau)\rvert^q 
      \lvert\mathrm{d}\zeta\rvert
    = \int_{\mathbb{R}} \frac{\lvert\Psi'\big(u+\mathrm{i}a(u)+\mathrm{i}\tau\big)
         	   \big(1+\mathrm{i}a'(u)\big)\rvert \,\mathrm{d}u}
          {\lvert\Psi\big(u+\mathrm{i}a(u)+\mathrm{i}\tau\big)
           	 - \alpha\rvert^q}\]
  is bounded and the boundary is independent of $\tau$. Thus, 
  $G\in H^q(\Omega_+)$, by definition.
\end{proof}

We could now prove that $T$ maps $H^p(\Omega_+)$ into $H^p(\mathbb{C}_+)$, 
which is the second part of our proof.
\begin{proposition}\label{pro-170601-1725}
  We have
  \[T\big(H^p(\Omega_+)\big)\subset H^p(\mathbb{C}_+).\]
\end{proposition}

\begin{proof}
  Fix $F\in H^p(\Omega_+)$, and let $f(z)= TF(z)$. 
  Since $F(w)$ has non-tangential boundary limit $F(\zeta)$ on $\Gamma$, and 
  \[F(w)
    = \frac{1}{2\pi\mathrm{i}} \int_{\Gamma}
        \frac{F(\zeta)}{\zeta-w}\,\mathrm{d}\zeta,\]
  by the remarks which follow Corollary~\ref{cor-170706-2100}, 
  then $f(z)$ has non-tangential boundary limit $f(x)$ on real axis, and 
  \[f(z)
    = F\big(\Phi(z)\big)\big(\Phi'(z)\big)^{\frac1p} 
    = \frac{1}{2\pi\mathrm{i}} \int_{\Gamma} 
          \frac{F(\zeta)\,\mathrm{d}\zeta}{\zeta-\Phi(z)}
          \big(\Phi'(z)\big)^{\frac{1}{p}}.\]
  Since 
  \begin{align*}
    \int_{\mathbb{R}} \lvert f(x)\rvert^p\,\mathrm{d}x
    &= \int_{\mathbb{R}} \big\lvert F\big(\Phi(x)\big)\big\rvert^p
        \cdot \lvert \Phi'(x)\rvert\,\mathrm{d}x           \\
    &= \int_{\Gamma} \lvert F(\zeta)\rvert^p \lvert\mathrm{d}\zeta\rvert
     \leqslant \lVert F\rVert_{H^p(\Omega_+)}^p
     < +\infty,
  \end{align*}
  we have $f(x)\in L^p(\mathbb{R},\,\mathrm{d}x)$. 
  For $z\in\mathbb{C}_+$, define 
  \[g(z)
    = \frac{1}{2\pi\mathrm{i}} \int_{\mathbb{R}}
        \frac{f(t)}{t-z}\,\mathrm{d}t,\]
  then $g(z)\in H^p(\mathbb{C}_+)$~\cite{02}. We are going to prove that
  $f(z)=g(z)$, which will finish the proof of the proposition. 
  Let $t=\Psi(\zeta)$ in the above integral, then 
  \[g(z)
    = \frac{1}{2\pi\mathrm{i}} \int_{\Gamma}
        \frac{f\big(\Psi(\zeta)\big)\Psi'(\zeta)}{\Psi(\zeta)-z}
        \,\mathrm{d}\zeta,\]
  and by $\Phi\big(\Psi(\zeta)\big)= \zeta$ or 
  $\Phi'\big(\Psi(\zeta)\big)\cdot \Psi'(\zeta)= 1$,
  \[f\big(\Psi(\zeta)\big)
    = F(\zeta)\big(\Psi'(\zeta)\big)^{-\frac{1}{p}},\]
  we have
  \[g(z)
    = \frac{1}{2\pi\mathrm{i}} \int_{\Gamma}
        \frac{F(\zeta)\big(\Psi'(\zeta)\big)^{\frac{1}{q}}}
           {\Psi(\zeta)-z}\,\mathrm{d}\zeta,\]
  then
  \[f(z)-g(z)
    = \frac{1}{2\pi\mathrm{i}} \int_{\Gamma}
        F(\zeta) \bigg(\frac{\big(\Phi'(z)\big)^{\frac{1}{p}}}{\zeta-\Phi(z)}
          - \frac{\big(\Psi'(\zeta)\big)^{\frac{1}{q}}}{\Psi(\zeta)-z}\bigg)
        \,\mathrm{d}\zeta.\]

  For fixed $z_0\in\mathbb{C}_+$, denote $\Phi(z_0)$ as 
  $w_0= u_0+\mathrm{i}a(u_0)+\mathrm{i}\tau_0$, 
  where $u_0$, $\tau_0\in\mathbb{R}$, then $\tau_0>0$, $w_0\in\Omega_+$, 
  $z_0=\Psi(w_0)$, $\Phi'(z_0)\cdot\Psi'(w_0)=1$, and there exists $\delta>0$, 
  such that open disk $D(w_0,\delta)\subset\Omega_+$. We could write
  \[f(z_0)-g(z_0)
    = \frac{1}{2\pi\mathrm{i}} \int_{\Gamma} F(\zeta)H(\zeta) \,\mathrm{d}\zeta,\]
  where
  \[H(\zeta)
    = \frac{\big(\Phi'(z_0)\big)^{\frac{1}{p}}}{\zeta-\Phi(z_0)}
      - \frac{\big(\Psi'(\zeta)\big)^{\frac{1}{q}}}{\Psi(\zeta)-z_0}
    = \frac{\big(\Psi'(w_0)\big)^{-\frac{1}{p}}}{\zeta-w_0}
      - \frac{\big(\Psi'(\zeta)\big)^{\frac{1}{q}}}{\Psi(\zeta)-\Psi(w_0)}.\]

  Define
  \[H(w)
    = \frac{\big(\Psi'(w_0)\big)^{-\frac{1}{p}}}{w-w_0}
      - \frac{\big(\Psi'(w)\big)^{\frac{1}{q}}}{\Psi(w)-\Psi(w_0)}
    = H_1(w)- H_2(w),\]
  for $w\in\Omega_+\setminus\{w_0\}$, then $H(w)$ is holomorphic since $\Psi$ is 
  a holomorphic representation from $\Omega_+$ onto $\mathbb{C}_+$, 
  and its boundary limit is $H(\zeta)$, where $\zeta\in\Gamma$.

  By the Laurent series expansion of holomorphic functions,
  \[\Psi(w)-\Psi(w_0)
    = \Psi'(w_0)(w-w_0)+\frac{\Psi''(w_0)}{2!}(w-w_0)^2+\cdots,\]
  for $w\in D(w_0,\delta)$, we have 
  $\Psi'(w)=\Psi'(w_0)+\Psi''(w_0)(w-w_0)+\cdots$, and
  \begin{align*}
    H_2(w)
    &= \frac{\big(\Psi'(w_0)+\Psi''(w_0)(w-w_0)+\cdots\big)^{\frac1q}}
         {\Psi'(w_0)(w-w_0)+\frac{\Psi''(w_0)}{2!}(w-w_0)^2+\cdots}     \\
    &= \frac{\big(\Psi'(w_0)\big)^{-\frac1p}}{w-w_0}
       \bigg(1+\frac{\Psi''(w_0)}{\Psi'(w_0)}(w-w_0)+\cdots\bigg)^{\frac1q}
       \bigg(1+\frac{\Psi''(w_0)}{2\Psi'(w_0)}(w-w_0)+\cdots\bigg)^{-1}   \\
    &= \frac{\big(\Psi'(w_0)\big)^{-\frac1p}}{w-w_0}
       \bigg(1+\frac{\Psi''(w_0)}{q\Psi'(w_0)}(w-w_0)+\cdots\bigg)
       \bigg(1-\frac{\Psi''(w_0)}{2\Psi'(w_0)}(w-w_0)+\cdots\bigg)        \\
    &= \frac{\big(\Psi'(w_0)\big)^{-\frac1p}}{w-w_0}
       \bigg(1+\Big(\frac1q-\frac12\Big)\frac{\Psi''(w_0)}{\Psi'(w_0)}(w-w_0)
         + \cdots\bigg)                                                  \\
    &= \frac{\big(\Psi'(w_0)\big)^{-\frac1p}}{w-w_0}
       + \Big(\frac1q-\frac12\Big)
           \frac{\Psi''(w_0)}{\big(\Psi'(w_0)\big)^{1+\frac1p}}
       +\cdots,
  \end{align*}
  then $H(w)$ could be extended holomorphically to $\Omega_+$, and
  \[H(w_0)
    = \Big(\frac12-\frac1q\Big)
      \frac{\Psi''(w_0)}{\big(\Psi'(w_0)\big)^{1+\frac1p}}.\]

  If we could prove that $H(w)\in H^q(\Omega_+)$, 
  then by Lemma~\ref{lem-170522-1424},
  \[f(z_0)-g(z_0)
    = \int_{\Gamma} F(\zeta)H(\zeta)\,\mathrm{d}\zeta= 0,\]
  and $f(z)= g(z)\in H^p(\mathbb{C}_+)$. For this reason, we are going to estimate
  \[I= \int_{\Gamma_{\tau}} \lvert H(w)\rvert^q\lvert\mathrm{d}w\rvert,
    \quad\text{for } \tau>0.\]

  Remember that $w_0= u_0+\mathrm{i}a(u_0)+\mathrm{i}\tau_0\in\Omega_+$, 
  $u_0\in\mathbb{R}$, $\tau_0>0$, and let $\delta_0>0$ be a small number 
  such that
  \[E_0= \bigl\{(u,v)\in\Omega_+\colon
      \lvert u-u_0\rvert\leqslant\delta_0,
      \ \big\lvert v-\big(a(u)+\tau_0\big)\big\rvert\leqslant\delta_0\bigr\}\]
  is contained in $D(w_0,\delta)$. For fixed $\tau>0$, define
  \[E_1= \{(u,v)\in\Gamma_{\tau}\colon \lvert u-u_0\rvert> \delta_0\},
    \quad E_2= \Gamma_{\tau}\setminus E_1,\]
  and we write
  \begin{align*}
    I
    &= \int_{E_1} \lvert H(w)\rvert^q\lvert\mathrm{d}w\rvert
       + \int_{E_2} \lvert H(w)\rvert^q\lvert\mathrm{d}w\rvert     \\
    &= I_1+ I_2.
  \end{align*}
  Then
  \begin{align*}
    \int_{E_1} \lvert H_1(w)\rvert^q\lvert\mathrm{d}w\rvert
    &= \frac{1}{\lvert\Psi'(w_0)\rvert^{\frac{q}{p}}} \int_{E_1}
    	    \frac{\lvert\mathrm{d}w\rvert}{\lvert w-w_0\rvert^q}       \\
    &\leqslant \frac{1}{\lvert\Psi'(w_0)\rvert^{\frac{q}{p}}} 
        \int_{\lvert u-u_0\rvert>\delta_0}
           \frac{\sqrt{1+M^2}\mathrm{d}u}{\lvert u-u_0\rvert^q}      \\
    &= \frac{2\delta_0^{1-q}\sqrt{1+M^2}}
            {(q-1)\lvert\Psi'(w_0)\rvert^{\frac{q}{p}}}.
  \end{align*}
  If $\lvert\tau-\tau_0\rvert>\delta_0$, then there exists $\delta_1>0$, such that
  $\lvert w-w_0\rvert> \delta_1$ for $w\in\Gamma_{\tau}$, and
  \begin{align*}
    \int_{E_2} \lvert H_1(w)\rvert^q\lvert\mathrm{d}w\rvert
    &= \frac{1}{\lvert\Psi'(w_0)\rvert^{\frac{q}{p}}} 
       \int_{E_2} \frac{\lvert\mathrm{d}w\rvert}{\lvert w-w_0\rvert^q}       \\
    &\leqslant \frac{\sqrt{1+M^2}}{\lvert\Psi'(w_0)\rvert^{\frac{q}{p}}} 
        \int_{\lvert u-u_0\rvert\leqslant\delta_0}
           \frac{\mathrm{d}u}{\delta_1^q}                      \\
    &= \frac{2\delta_0\sqrt{1+M^2}}
            {\delta_1^q\lvert\Psi'(w_0)\rvert^{\frac{q}{p}}}.
  \end{align*}

  Since $\Psi$ is a holomorphic representation from $\Omega_+$ onto $\mathbb{C}_+$, 
  $\Psi(E_0)$ is an open set in $\mathbb{C}_+$, and there exists $\varepsilon_0>0$, 
  such that $\lvert\Psi(w)-\Psi(w_0)\rvert\geqslant \varepsilon_0$ for
  $w\in\Omega_+\setminus E_0$. For fixed $\tau$, denote 
  $E_{\tau}=\{u\in\mathbb{R}\colon 
     u+ \mathrm{i}a(u)+ \mathrm{i}\tau\notin E_0\}$, 
  then by Lemma~\ref{lem-170522-1536},
  \[\int_{E_\tau} \lvert H_2(w)\rvert^q\lvert\mathrm{d}w\rvert
    = \int_{E_{\tau}}
         \frac{\big\lvert\Psi'\big(u+\mathrm{i}a(u)+\mathrm{i}\tau\big)
          	     \big(1+\mathrm{i}a'(u)\big)\big\rvert}
              {\big\lvert\Psi\big(u+\mathrm{i}a(u)+\mathrm{i}\tau\big)
                 - \Psi(w_0)\big\rvert^q}\,\mathrm{d}u\]
  is bounded, and the boundary is independent of $\tau$.
  
  Thus, by Minkowski's inequality, we have showed that, $I_1$ for all $\tau>0$, 
  and $I_2$ for $\lvert\tau-\tau_0\rvert> \delta_0$, are bounded, 
  and the boundaries are independent of $\tau$. Since $H(w)$ is continuous on 
  compact set $E_0$, we could denote 
  $M_2=\max\{\lvert H(w)\rvert\colon w\in E_0\}$, 
  then for $|\tau-\tau_0|\leqslant\delta_0$,
  \begin{align*}
    I_2
    &= \int_{E_2} \big\lvert H\big(u+\mathrm{i}a(u) 
            +\mathrm{i}\tau\big)\big\rvert^q
         \cdot \lvert 1+\mathrm{i}a'(u)\rvert \,\mathrm{d}u                \\
    &\leqslant M_2^q\cdot \sqrt{1+M^2}\cdot 2\delta_0.
  \end{align*}
  It follows that $I$ is bounded, independent of $\tau$, for all $\tau>0$, 
  and $H(w)\in H^q(\Omega_+)$, then we have proved that
  \[TF(z)=f(z)=g(z)\in H^p(\mathbb{C}_+).\]
  Thus, $T(H^p(\Omega_+))\subset H^p(\mathbb{C}_+)$.
\end{proof}

Besides, by Fatou's Lemma, 
\[\lVert f\rVert_{H^p(\mathbb{C}_+)}^p
  = \int_{\mathbb{R}} \lvert f(t)\rvert^p \,\mathrm{d}t
  = \int_{\Gamma} \lvert F(\zeta)\rvert^p \lvert\mathrm{d}\zeta\rvert
  \leqslant \lVert F\rVert_{H^p(\Omega_+)}^p,\]
that is, $\lVert T\rVert\leqslant 1$.

\begin{lemma}\label{lem-170601-1652}
  Suppose that $p>0$, $F\in H^p(\Omega_+)$, and $\tau>0$, then
  \[|F(\zeta+\mathrm{i}\tau)|
    \leqslant C\lVert F\rVert_{H^p(\Omega_+)}\tau^{-\frac1p},\quad
    \text{for }\zeta\in\Gamma,\]
  where $C=(\frac{2+2M^2}{\pi})^{\frac1p}$.
\end{lemma}

\begin{proof}
  Let $w_0=\zeta+\mathrm{i}\tau=u_0+\mathrm{i}v_0\in\Omega_+$, 
  $d=\frac{\tau}{\sqrt{1+M^2}}$, and 
  $D=\{w\in\mathbb{C}\colon |w-w_0|<d\}$, then $D\subset\Omega_+$. 
  We know that $|F|^p$ is subharmonic since $F$ is holomorphic on $\Omega_+$, 
  hence
  \begin{align*}
    |F(w_0)|^p
    &\leqslant \frac{1}{\pi d^2} \iint_{|w-w_0|<d} 
        |F(w)|^p\,\mathrm{d}\lambda(w)            \\
    &\leqslant \frac{1}{\pi d^2} \int_{v_0-\tau}^{v_0+\tau}\!\!\!
        \int_{\Gamma} |F(\zeta+\mathrm{i}v)|^p \,\mathrm{d}s\,\mathrm{d}v   \\
    &\leqslant \frac{2\tau}{\pi d^2}\lVert F\rVert_{H^p(\Omega_+)}^p   
     = \frac{2(1+M^2)}{\pi\tau} \lVert{F}\rVert_{H^p(\Omega_+)}^p,
  \end{align*}
  where $\mathrm{d}\lambda$ is the area measure on $\mathbb{C}$, 
  it follows that 
  \[|F(\zeta+\mathrm{i}\tau)| 
    \leqslant \Big(\frac{2+2M^2}{\pi}\Big)^{\frac1p} 
        \lVert F\rVert_{H^p(\Omega_+)}\tau^{-\frac1p},\]
  for $\zeta\in\Gamma$ and $\tau>0$.
\end{proof}

\begin{lemma}
  $H^p(\Omega_+)$ is complete, or $H^p(\Omega_+)$ is a Banach space.
\end{lemma}

\begin{proof}
  Suppose $\{F_n(w)\}$ is a Cauchy sequence in $H^p(\Omega_+)$, that is 
  \[\lVert F_n-F_m\rVert_{H^p(\Omega_+)}\to 0,\quad
    \text{as }n,m\to\infty.\]

  If $\tau>0$, then, by Lemma~\ref{lem-170601-1652}, 
  for $w=\zeta+\mathrm{i}\tau\in\Omega_+$ with $\zeta\in\Gamma$ and $\tau>0$,
  \[\lvert F_n(w)- F_m(w)\rvert^p
    \leqslant \frac{C}{\tau} \lVert F_n-F_m\rVert_{H^p(\Omega_+)}^p,\]
  where $C$ is a positive constant. It follows that $\{F_n(w)\}$ converges
  on compact sets in $\Omega_+$, and we could assume that 
  $\{F_n(w)\}$ converges to a holomorphic function $F(w)$ on $\Omega_+$.

  For any $\varepsilon>0$, there exists $m\in\mathbb{N}$, such that if $n>m$, 
  then $\lVert F_n-F_m\rVert_{H^p(\Omega_+)}<\varepsilon$. By Fatou's lemma,
  \[\int_{\Gamma}\lvert F(\zeta+\mathrm{i}\tau)
        - F_m(\zeta+\mathrm{i}\tau)\rvert^p \lvert\mathrm{d}\zeta\rvert
    \leqslant \lim_{n\to\infty} \int_{\Gamma}\lvert F_n(\zeta+\mathrm{i}\tau)
        - F_m(\zeta+\mathrm{i}\tau)\rvert^p \lvert\mathrm{d}\zeta\rvert
    \leqslant \varepsilon^p,\]
  or $\lVert F-F_m\rVert_{H^p(\Omega_+)}\leqslant\varepsilon$, 
  then $F(w)\in H^p(\Omega_+)$ as 
  $\lVert F\rVert_{H^p(\Omega_+)}
    \leqslant \varepsilon+ \lVert F_m\rVert_{H^p(\Omega_+)}$.
  Besides, $\lVert F-F_n\rVert_{H^p(\Omega_+)}<2\varepsilon$ for all $n>m$, 
  which means that $H^p(\Omega_+)$ is complete.
\end{proof}

\begin{proposition}\label{pro-170601-1730}
  $T^{-1}\colon H^p(\mathbb{C}_+)\to H^p(\Omega_+)$ is bounded.
\end{proposition}

\begin{proof}
  Since both $H^p(\mathbb{C}_+)$ and $H^p(\Omega_+)$ are Banach spaces,
  the linear transform $T\colon H^p(\Omega_+)\to H^p(\mathbb{C}_+)$
  is one-to-one, onto and bounded, then $T^{-1}$ is also bounded, 
  by the open mapping theorem.
\end{proof}

The proof of Theorem~\ref{thm-170601-1721} is now completed once we combine 
the proofs of Proposition~\ref{pro-170522-1450}, 
Proposition~\ref{pro-170601-1725}, and Proposition~\ref{pro-170601-1730}.

\section*{Funding}
This work is supported by National Natural Science Foundation 
of China(Grant No.\@ 11271045).


\begin{thebibliography}{99}
\bibitem{01} Meyer Y, Coifman R. 
  Wavelets: Calder\'on-Zygmund and Multilinear Operators. 
  Cambridge (UK): Cambridge University Press; 1997.

\bibitem{03} Duren PL. Theory of $H^p$ Spaces. 
  New York: Academic Press; 1970.

\bibitem{06} Garnett JB. Bounded Analytic Functions, 
  New York: Springer; 2007.

\bibitem{02} Deng GT. Complex Analysis (in Chinese), 
  Beijing: Beijing Normal University Press; 2010.
  
\bibitem{05} Wang CQ. Growth Estimates for Some Subharmonic 
  Functions on Tubular Regions [Dissertation] (in Chinese), 
  Beijing: Beijing Normal University; 2017.

\bibitem{04} Saks C, Zygmund A. Analytic Functions. 
  Poland: Polish Scientific Publishers; 1971.

\bibitem{07} Kenig C. Weighted $H^p$ spaces on Lipschitz domains. 
  Amer. J. Math. 1980;102:129-163.
  
\end{thebibliography}
\end{document}